\documentclass[12pt,a4paper]{article}
\usepackage{amsmath,amssymb,amsthm,graphicx}
\usepackage[left=1in,right=1in,top=1in,bottom=1in]{geometry}
\usepackage{cite}
\usepackage[british]{babel}
\usepackage{enumitem}
\usepackage{hyperref} 

\hypersetup{
    colorlinks=false,
    pdfborder={0 0 0},
}

\newtheorem{theorem}{Theorem}[section]
\newtheorem{corollary}[theorem]{Corollary}
\newtheorem{proposition}[theorem]{Proposition}
\newtheorem{lemma}[theorem]{Lemma}
 
 \newtheorem{problem}[theorem]{Problem}

\numberwithin{equation}{section}

\theoremstyle{definition}
\newtheorem{definition}[theorem]{Definition}
\newtheorem{examplex}[theorem]{Example}

\newenvironment{example}
  {\pushQED{\qed}\examplex}
  {\popQED\endexamplex}

\theoremstyle{remark}
\newtheorem{remark}[theorem]{Remark}
\newtheorem{remarks}[theorem]{Remarks}
\newtheorem*{remark*}{Remark}

 \allowdisplaybreaks

\newcommand{\1}[1]{{\mathbf 1}{\{#1\}}}

\newcommand{\R}{{\mathbb R}}
\newcommand{\Z}{{\mathbb Z}}
\newcommand{\Q}{{\mathbb Q}}
\newcommand{\N}{{\mathbb N}}
\newcommand{\ZP}{{\mathbb Z}_+}
\newcommand{\RP}{{\mathbb R}_+}
\newcommand{\Sp}[1]{{\mathbb S}^{#1}}

\newcommand{\overbar}[1]{\mkern 1.5mu\overline{\mkern-1.5mu#1\mkern-1.5mu}\mkern 1.5mu}
\newcommand{\bR}{\overbar{{\mathbb R}^d}}

\DeclareMathOperator{\Exp}{\mathbb{E}}
\renewcommand{\Pr}{{\mathbb P}}

\DeclareMathOperator{\Int}{int} 
\newcommand{\sint}{\mathop \mathrm{s\text{-}int}}

\DeclareMathOperator{\sign}{sgn}

\newcommand{\hull}{\mathop \mathrm{hull}}
\newcommand{\shull}{\mathop \mathrm{s\text{-}hull}}
\newcommand{\cl}{\mathop \mathrm{cl}}
\newcommand{\supp}{\mathop \mathrm{supp}}
\newcommand{\diam}{\mathop \mathrm{diam}}

\newcommand{\eps}{\varepsilon}

\newcommand{\re}{{\mathrm{e}}}
\newcommand{\rc}{{\mathrm{c}}}

\newcommand{\cA}{{\mathcal A}}

\newcommand{\cC}{{\mathcal C}}
\newcommand{\cD}{{\mathcal D}}
\newcommand{\cE}{{\mathcal E}}

\newcommand{\cH}{{\mathcal H}}

\newcommand{\cL}{{\mathcal L}}

\newcommand{\cP}{{\mathcal P}}
\newcommand{\cR}{{\mathcal R}}
\newcommand{\cS}{{\mathcal S}}
\newcommand{\cX}{{\mathcal X}}

\newcommand{\as}{\ \text{a.s.}}
\newcommand{\io}{\ \text{i.o.}}

\newcommand{\tod}{\overset{\mathrm{d}}{\longrightarrow}}

\newcommand{\bx}{{\mathbf{x}}}
\newcommand{\by}{{\mathbf{y}}}
\newcommand{\bz}{{\mathbf{z}}}
\newcommand{\bu}{{\mathbf{u}}}

\newcommand{\bv}{{\mathbf{v}}}
\newcommand{\bw}{{\mathbf{w}}}
\newcommand{\be}{{\mathbf{e}}}
\newcommand{\0}{{\mathbf{0}}}

\newcommand{\eqd}{\overset{d}{=}}

\newcommand{\Rad}{\text{Rad}}

\makeatletter
\def\namedlabel#1#2{\begingroup  
    (#2)%
    \def\@currentlabel{#2}%
    \phantomsection\label{#1}\endgroup
}
\makeatother

\newlist{myenumi}{description}{10}
\setlist[myenumi]{labelindent=\parindent, leftmargin=*, align=left, itemsep=1pt, parsep=0pt}
\setlist[myenumi]{leftmargin=0pt}

\begin{document}

\title{Angular asymptotics for random walks}
\author{Alejandro L\'opez Hern\'andez\footnote{\texttt{\href{mailto:alejandro.lopez-hernandez20@imperial.ac.uk}{alejandro.lopez-hernandez20@imperial.ac.uk}}} \and Andrew R.\ Wade\footnote{\raggedright Durham University, Department of Mathematical Sciences, South Road, DH1 3LE.}~\footnote{\texttt{\href{mailto:andrew.wade@durham.ac.uk}{andrew.wade@durham.ac.uk}}}}

\date{\today}
\maketitle

\begin{abstract}
We study the set of directions asymptotically explored by a spatially homogeneous random walk
in $d$-dimensional Euclidean space.
We survey some pertinent results of Kesten and Erickson,
make some further observations, and present some examples.
We also explore links to the asymptotics of one-dimensional projections,
and to the growth of the convex hull of the random walk.
  \end{abstract}

\medskip


\medskip


\section{Introduction}
\label{sec:intro}

In this paper we examine some aspects of the way in which a random walk in $d$ dimensions explores space, specifically through
the limit points of the trajectory projected onto the sphere, and related questions concerning the growth of the convex hull of the walk.
We ask, roughly speaking, in which directions does the walk grow without bound?

Let $d \in \N := \{1,2,3,\ldots\}$. Let $X, X_1, X_2, \ldots$ be i.i.d.~random variables in $\R^d$,
and define the associated random walk $(S_n ; n \in \ZP)$ by $S_0 := \0$
and $S_n = \sum_{k=1}^n X_k$ for $n \geq 1$; here and subsequently $\0$
is the origin in $\R^d$ and $\ZP := \{0,1,2,\ldots\}$.
We suppose throughout that $S_n$ is genuinely $d$-dimensional, i.e.,
$\supp X$ is not contained in a $(d-1)$-dimensional subspace of $\R^d$.

Denote by $\bx \cdot \by$ the Euclidean inner product of vectors 
$\bx, \by \in \R^d$, and by $\| \, \cdot \, \|$ the Euclidean norm on $\R^d$. 
Set $\Sp{d-1} := \{ \bx \in \R^d : \| \bx \| = 1 \}$.
For $\bx \in \R^d \setminus \{ \0 \}$ define $\hat \bx := \bx / \| \bx \|$;
we also set $\hat \0 := \0$. We view vectors in $\R^d$ as column vectors where necessary. Whenever the appropriate expectation exists, we write $\mu := \Exp X$, so $\mu \in \R^d$ is the mean drift
vector of the random walk.

In Section~\ref{sec:directions} we look at the limit points in $\Sp{d-1}$ of the sequence
$\hat S_0, \hat S_1, \ldots$, drawing on closely related work of Kesten and Erickson~\cite{kesten70,ek,erickson76,erickson00}.
In particular, an adaptation of an idea of Kesten shows that the limit set is a.s.~equal to a deterministic closed $\cD \subseteq \Sp{d-1}$ (see Theorem~\ref{thm:D}).
In Section~\ref{sec:comments} we make more explicit the connection to the work of Kesten and Erickson~\cite{kesten70,ek,erickson76,erickson00}
on limit sets graded by particular speeds of growth. Section~\ref{sec:direction}
considers the special case where $\cD$ has a single element, in which the walk
is transient with a limiting direction. In Section~\ref{sec:plane}
we make some observations about the case where the walk has increments
with mean zero (zero drift). Section~\ref{sec:erickson}
presents an argument due to Erickson which shows that
an arbitrary closed $\cD \subseteq \Sp{d-1}$
can be achieved as the limit set by constructing a random walk with suitable heavy-tailed increments (Theorem~\ref{thm:heavytails}).
In Section~\ref{sec:convexity} we   introduce some relevant convexity ideas. Section~\ref{sec:projections}
turns to considering the asymptotics of the one-dimensional projections $S_n \cdot \bu$, $\bu \in \Sp{d-1}$.
Section~\ref{sec:hull} studies the convex hull of the trajectory, and draws
some connections to the preceding sections. In Section~\ref{sec:examples} we present some examples.
These illustrate, for instance, that while walks whose increments are symmetric
and have zero mean must have $\cD = \Sp{d-1}$ when $d=2$ (Proposition~\ref{prop:plane}),
for $d \geq 4$ the set $\cD$ can have measure zero in $\Sp{d-1}$ (Example~\ref{ex:four-dim}).

We make a few historical comments.
As observed by Blackwell, and Chung and Derman (see \cite[p.~493]{hs} and~\cite[p.~658]{blackwell}), it is a consequence of the Hewitt--Savage
zero--one law that $\Pr ( S_n \in A \io ) \in \{0,1\}$
for any Borel set $A \subseteq \R^d$. 
Those authors raised the question of classifying sets $A$ accordingly for a given random walk (see e.g.~\cite[p.~447]{cd}).
For bounded sets $A$ containing the origin in their interior, the question is that of recurrence vs.~transience, and is answered by Chung and Fuchs~\cite{cf}.

Attention focused on determining infinite sets $A$ visited infinitely often by (transient) random walks on $\Z^d$, $d \geq 3$, most notably for the case where the random walk converges to Brownian motion,
where a classification of recurrent sets~$A$ is available in the form of `Wiener's test':
for the case of simple symmetric random walk, see~\cite{ik,bucythesis,bucy},
for bounded and symmetric increments, see~\cite[\S 6.5]{ll}, and for increments with
 zero mean and finite second moments, see~\cite{jo,spitzer,uchiyama}. 
Wiener's test and its generalizations~\cite{bucy,lamp,murdoch} give analytic criteria in terms of the
capacity of~$A$ or Green's functions of the walk.
An early paper of Doney~\cite{doney} showed that Wiener's test can yield very useful
information, 
but, according to Spitzer, ``in general the computations are prohibitively difficult''~\cite[p.~320]{spitzer}.
The present paper addresses questions related to the transience or recurrence of sets $A$
that are cones or half-spaces.

\section{Recurrent directions}
\label{sec:directions}

We say $\bu \in \Sp{d-1}$ is a \emph{recurrent direction} for $S_n$
if the sequence $\hat S_n$ has an accumulation point at $\bu$, i.e.,
if $\hat S_n$ has $\bu$ as a subsequential limit.
Let $L$
be the (random) set of all recurrent directions for $S_n$;
equivalently, 
\[ L := \{ \bu \in \Sp{d-1} : \liminf_{n\to\infty} \| \hat S_n - \bu \| = 0 \} .\]
Note that in $L$ the possible accumulation point at $\0$ is excluded.
Also define
\[ \cD := \{ \bu \in \Sp{d-1} : \liminf_{n \to \infty} \| \hat S_n - \bu \| = 0, \as \} , \]
i.e., the set of all a.s.~recurrent directions for $S_n$. 

For $d=1$, ruling out the degenerate case where $\Pr ( X = 0 ) = 1$,
the well known trichotomy (see e.g.~\cite[Theorem~4.1.2]{dur})
states that either (i) $S_n \to +\infty$, a.s.,
(ii) $S_n \to -\infty$, a.s., or
(iii) $\liminf_{n \to \infty} S_n = -\infty$ and  $\limsup_{n \to \infty} S_n = +\infty$, a.s.,
corresponding to (i) $\cD = \{ +1\}$, (ii) $\cD = \{ -1\}$, and~(iii) $\cD = \{-1,+1\}$ (this latter case
includes the case where $S_n$ is recurrent). Our primary interest here is the case $d \geq 2$.

The following result is a consequence of a more general statement of
 Erickson~\cite{erickson76} (see also~\S\ref{sec:comments} below),
who pointed out that it can be obtained
by adapting an argument of Kesten~\cite{kesten70} (see also Lemma~1 of~\cite{kuelbs}
for a generalization attributed to Neidhardt). An alternative proof of
the fact that $L$ is deterministic could be obtained by appealing to a general
 zero--one result for random closed sets such as Proposition~1.1.30 of~\cite{molchanov},
having first established that $L$ is closed.

\begin{theorem}
\label{thm:D}
The set $\cD$ is a non-empty, closed subset of $\Sp{d-1}$, and $\Pr (L  = \cD) = 1$.
\end{theorem}

 We work towards the proof of Theorem~\ref{thm:D}. 
For $\bu \in \Sp{d-1}$ and $r > 0$, define the set
\[ C ( \bu ; r ) := \{ \bx \in \R^d \setminus \{ \0\} : \| \hat \bx - \bu \| < r \} \]
and the event
\[ A ( \bu; r) := \{ S_n \in C (\bu; r ) \text{ i.o.} \} .\]
By the Hewitt--Savage zero--one law (see e.g.~\cite[Theorem~4.1.1]{dur}), $\Pr ( A (\bu ; r) ) \in \{ 0, 1\}$.

Let  $B(\bx;r) := \{ \by \in \R^d : \| \bx - \by \| < r \}$ denote the open Euclidean ball centred at $\bx\in \R^d$ with radius $r > 0$, and
for $\bu \in \Sp{d-1}$
let $B_s (\bu ; r) := \Sp{d-1} \cap B ( \bu ; r)$. For $A \subseteq \R^d$, we write $\cl A$ for the closure of $A$ in $\R^d$ in the usual topology.

\begin{lemma}
\label{lem:cones}
 For any $\bu \in \Sp{d-1}$ and any $r >0$,
we have
\[ \{ L \cap B_s (\bu ; r) \neq \emptyset \} \subseteq A (\bu ; r) \subseteq \{ L \cap \cl B_s (\bu ; r) \neq \emptyset \} .\]
\end{lemma}
\begin{proof}
First note that
\[  A ( \bu ; r ) = \{ \hat S_n \in B_s (\bu ; r) \io \} . \]
Hence $A (\bu; r)$ implies that $\hat S_n \in \cl B_s ( \bu ; r)$ i.o.,
and since $\cl B_s ( \bu ; r)$ is compact, $\hat S_n$ must have an accumulation point
in $\cl B_s ( \bu ; r)$. On the other hand,
if $\hat S_n$ has an accumulation point in $B_s (\bu ; r)$,
then since $B_s (\bu ; r)$ is open in $\Sp{d-1}$ we have $\hat S_n \in B_s (\bu ; r)$ i.o.
\end{proof}

The following continuity property is a key ingredient in the proof of Theorem~\ref{thm:D}.

\begin{lemma}
\label{lem:limsup}
Given any sequence $\bx_1, \bx_2, \ldots \in \Sp{d-1}$,
and any $\bu, \bv \in \R^d$,
\[ \left| \limsup_{n \to \infty} ( \bx_n  \cdot \bu ) - \limsup_{n \to \infty} ( \bx_n  \cdot \bv ) \right| \leq \| \bu - \bv \| .\]
\end{lemma}
\begin{proof}
Suppose that $\bu, \bv \in \R^d$.
Then
\begin{align*} \limsup_{n \to \infty} ( \bx_n  \cdot \bv )
& \leq \limsup_{n \to \infty} ( \bx_n \cdot \bu )
+ \limsup_{n \to \infty} ( \bx_n \cdot ( \bv - \bu ) )\\
& \leq  \limsup_{n \to \infty} ( \bx_n \cdot \bu ) + \| \bv - \bu \| ,\end{align*}
since $\| \bx_n \| = 1$. With a  similar argument in the other direction, we get
the result.
\end{proof}

\begin{lemma}
\label{lem:closed}
The set $\cD$ is closed in $\Sp{d-1}$.
\end{lemma}
\begin{proof}
Note that for any $\bu \in \Sp{d-1}$,
\[
 \| \hat S_n - \bu \|^2 = (\hat S_n - \bu ) \cdot (\hat S_n - \bu  ) = 1 + \1{ S_n \neq \0} - 2 \hat S_n \cdot \bu ,\]
so that
\[
\liminf_{n \to \infty} \| \hat S_n - \bu \| =0 \text{ if and only if } \limsup_{n \to \infty} (\hat S_n \cdot \bu ) = 1. \]
Thus
\[ \Sp{d-1} \setminus \cD = \bigl\{ \bu \in \Sp{d-1} : \limsup_{n \to \infty} ( \hat S_n \cdot \bu ) < 1 \as \bigr\} .\]
Consider $\bu \in \Sp{d-1} \setminus \cD$. By the Hewitt--Savage theorem,
$\limsup_{n \to \infty} ( \hat S_n \cdot \bu ) = c$ a.s.~for a constant $c <1$.
Lemma~\ref{lem:limsup} shows that for any $\bv \in \Sp{d-1}$ with $\| \bu - \bv \| \leq \frac{1-c}{2}$, a.s.,
\[ \limsup_{n \to \infty} ( \hat S_n \cdot \bv ) \leq c + \frac{1-c}{2} = \frac{1+c}{2} <1 , \]
so that $\bv \in \Sp{d-1} \setminus \cD$.
Thus $\Sp{d-1} \setminus \cD$ is open in $\Sp{d-1}$.
\end{proof}

Now we can complete the proof of Theorem~\ref{thm:D}.

\begin{proof}[Proof of Theorem~\ref{thm:D}.]
We adapt, in part, an argument from the proof of Theorem~1 of~\cite{kesten70}.
We call a ball $B_s (\bu ; r)$ \emph{rational} if 
$\bu \in \Sp{d-1} \cap \Q^d$ and $r \in \Q \cap (0,\infty)$.
Note that $\Sp{d-1} \cap \Q^d$ is dense in $\Sp{d-1}$, as follows from an argument based on stereographic projection (see e.g.~\cite{schmutz}).
Let $\cR$ denote the (countable) set of all rational balls, and set
\[ \cC := \{   B \in \cR : \Pr ( \hat S_n \in B \io ) =1 \} .\]
Then since $\cR$ is countable, and, by the Hewitt--Savage
theorem, $\Pr ( \hat S_n \in B \io ) \in \{0,1\}$
for any $B \in \cR$, we have
\begin{equation}
\label{eq0}
 \Pr ( \hat S_n \in B \io \text{ for all } B \in \cC \text{ but for no } B \in \cR \setminus \cC ) = 1.
\end{equation}
Observe that
\begin{equation}
\label{eq1a}
\bu \in L \text{ if and only if }
\hat S_n \in B \io \text{ for every } B \in \cR \text{ with } \bu \in B, 
\end{equation}
and so
 $\bu \in \cD$ if and only if
\begin{equation}
\label{eq1}
 \Pr ( \hat S_n \in B \io \text{ for every } B \in \cR \text{ with } \bu \in B ) = 1 . \end{equation}
In particular, if  $B \in \cR$ contains some $\bu \in \cD$, then $B \in \cC$.
With~\eqref{eq0}, this means that
\[ \Pr ( \text{for all } \bu \in \cD, \, \hat S_n \in B \io \text{ for every } B \in \cR \text{ with } \bu \in B ) = 1.\]
Together with~\eqref{eq1a}, it follows that $\Pr ( \cD \subseteq L ) =1$.
 
Let $\cC_k$ be the set of $B \in \cC$ with $\diam B < 1/k$.
Let $W_k := \cup \cC_k$ and $W := \cap_{k \in \N} W_k$.
Then it follows from~\eqref{eq1} that $\bu \in \cD$ if and only
if for every $k \in \N$ there exists some $B \in \cC_k$ with $\bu \in B$.
That is, $\bu \in \cD$ if and only if $\bu \in W$, i.e., $\cD = W$.

Let $\cR_k$ be the set of $B \in \cR$ with $\diam B < 1/k$.
Now let $M_k := \cup \{ B \in \cR_k : L  \cap B \neq \emptyset \}$.
Let $B \in \cR$. Since $B$ is open in $\Sp{d-1}$, we have that
$B \cap L  \neq \emptyset$
implies that $\hat S_n \in B$ i.o.
So $M_k \subseteq \cup \{ B \in \cR_k : \hat S_n \in B \io \}$.
Hence by~\eqref{eq0} we have that
$\Pr ( M_k \subseteq \cup \cC_k) =1$, i.e., $\Pr ( M_k \subseteq W_k ) =1$.
It follows that $\Pr ( \cap_{k \in \N} M_k \subseteq \cD ) =1$.
Note that if $\bu \in L$, then for all $k \in \N$ we have $B \cap L \neq \emptyset$ for some $B \in \cR_k$,
so $\bu \in M_k$ for all $k$; hence $L \subseteq \cap_{k\in \N} M_k$ a.s.
Hence we conclude that $\Pr ( L \subseteq \cD ) =1$.

To prove that $\cD$ is non-empty,
taking $r=2$ in Lemma~\ref{lem:cones} shows that $\hat S_n$ has
at least one accumulation point in $L$, since $C (\bu; 2 ) = \R^d \setminus \{ \0 \}$
and, since $S_n$ is genuinely $d$-dimensional, $S_n \neq \0$ i.o., a.s.
\end{proof}

Here is an alternative characterization of the set $\cD$.

\begin{proposition}
\label{prop:D-cones}
We have that 
\[ \cD = \{\bu\in\Sp{d-1}: \Pr( A(\bu;r))=1 \text{ for all } r>0\}. \]
\end{proposition}
\begin{proof}
Define the set $\cD' = \{\bu\in\Sp{d-1}: \Pr( A(\bu;r))=1 \text{ for all } r>0\}$.
If $\bu\in\cD'$, then $\Pr(A(\bu;1/m))=1$ for all $m \in \N$, and so 
 $\Pr(\cap_{m=1}^{\infty}A(\bu;1/m))=1$. In particular,
$\Pr ( \hat S_n \in B_s (\bu;1/m) \io \text{ for all } m\in\N ) =1$. 
In other words, a.s.,
 $\liminf_{n\to \infty}\|\hat S_n-\bu\|<1/m$ for all $m\in \N$, and hence $\liminf_{n\to\infty}\|\hat S_n -\bu\|=0$, a.s., so $\bu\in\cD$. 
Thus $\cD'\subseteq\cD$.

On the other hand, suppose that $\bu\in \Sp{d-1} \setminus \cD'$. Then there exists $r>0$ such that $\Pr(A(\bu;r))<1$, and, by the Hewitt--Savage theorem, 
in fact $\Pr(A(\bu;r))=0$.   Lemma~\ref{lem:cones} shows that
$A(\bu;r)^\rc\subseteq\{L\cap B_s(\bu;r)=\emptyset\}$ and hence $\Pr(L\cap B_s(\bu;r)=\emptyset)=1$.
In particular, this means that $\Pr(\bu\in L)=0$ and so $\bu\notin \cD$. This shows that $\cD\subseteq\cD'$.
\end{proof}

We next show that the recurrent directions are determined solely
by the behaviour of the walk at increasingly large distances from the origin.
Define 
\begin{equation}
\label{eq:asymptotic-directions}
 L_\infty := \left\{ \bu \in \Sp{d-1} : \liminf_{n \to \infty} \left( \frac{1}{1+\| S_n \|} + \| \hat S_n - \bu \| \right) = 0  \right\} , \end{equation}
and  
\[ \cD_\infty :=  \left\{ \bu \in \Sp{d-1} : \liminf_{n \to \infty} \left( \frac{1}{1+\| S_n \|} + \| \hat S_n - \bu \| \right) = 0, \as  \right\} .\]
In other words, $\bu \in L_\infty$ if and only if 
there exists a (random) subsequence $n_k$ of $\ZP$ such that both $\lim_{k \to \infty} \| S_{n_k} \| = \infty$
and
$\lim_{k \to \infty} \hat S_{n_k} = \bu$. 
If $\bu \in L_\infty$ we say that $\bu$ is an \emph{asymptotic direction} for the random walk.
Clearly an asymptotic direction is a recurrent direction, so
$\Pr( L_\infty \subseteq L ) = 1$ and $\cD_\infty \subseteq \cD$.

\begin{proposition}
\label{prop:recurrence}
If $S_n$ is recurrent, then $\cD = \cD_\infty = \Sp{d-1}$ and $\Pr ( L = L_\infty = \Sp{d-1} ) = 1$.
\end{proposition}
\begin{proof}
Suppose that $S_n$ is recurrent.
Since $\cD_\infty \subseteq \cD$ and $L_\infty \subseteq L$, 
it suffices to show that $\cD_\infty = \Sp{d-1}$ and $\Pr ( L_\infty = \Sp{d-1} ) = 1$.
  Proposition~\ref{prop:recurrence-coverage} shows that there is some $h \in (0,\infty)$ such that,
	a.s., for every $\bx \in \R^d$,
	$S_n \in B ( \bx ; h)$ i.o. But for every $\bu \in \Sp{d-1}$, every $r>0$, and every $R \in (h,\infty)$,
	$C ( \bu ; r)$ contains some $B ( \bx ; h)$ with $\| \bx \| > 2 R$, so that, a.s.,
	 for every $\bu \in \Sp{d-1}$, every $r>0$, and every $R \in (h,\infty)$, 
	there is a subsequence $n_k$ along which
	$\| \hat S_{n_k} - \bu \| < r$ and $\| S_{n_k} \| > R$. This shows that $\Pr ( L_\infty = \Sp{d-1} ) = 1$,
	and essentially the same argument implies that $\cD_\infty = \Sp{d-1}$.
\end{proof}

\begin{corollary}
\label{cor:not-full-sphere}
If $\cD \neq \Sp{d-1}$, then $S_n$ is transient.
\end{corollary}

The next result says that, a.s., the sets of recurrent and asymptotic directions coincide. 

\begin{theorem}
\label{thm:D-infty}
We have $\cD_\infty = \cD$, and $\Pr ( L_\infty = \cD ) = 1$.
\end{theorem}
\begin{proof}
The recurrent case is contained in Proposition~\ref{prop:recurrence};
thus suppose that $S_n$ is transient. Then since $\| S_n \| \to \infty$ a.s.,
we have that $\Pr ( L = L_\infty ) =1$ and $\cD = \cD_\infty$.
Combined with Theorem~\ref{thm:D}, this gives the result.
\end{proof}

Next we show how a distributional limit gives rise to recurrent directions.
Here and elsewhere, `$\tod$' denotes convergence in distribution and `$\supp$'
denotes the support of an $\R^d$-valued
random variable.

\begin{proposition}
\phantomsection
\label{prop:D-supp}
\begin{itemize}
\item[(i)]
Suppose that there is a random vector $\zeta \in \Sp{d-1}$ such that
$\hat S_n \tod \zeta$ as $n \to \infty$. Then $\supp \zeta \subseteq \cD$.
\item[(ii)]
Suppose there is a sequence $a_n$ of positive real numbers 
and a random vector $\xi \in \R^d$ with $\Pr ( \xi = \0 ) = 0$
such that 
$S_n /a_n \tod \xi$ as $n \to \infty$.
Then
$\supp \hat \xi \subseteq \cD$.
\end{itemize}
\end{proposition}
\begin{proof}
For part~(i), suppose that $\hat S_n \tod \zeta$.
Then, for a given $\bu \in \Sp{d-1}$,
for all but countably many $\eps >0$,
\begin{align*}
\Pr(\| \hat S_n -\bu\|<\eps \text{ i.o.}) 
& = \Pr \left(\bigcap_{n=1}^{\infty}\bigcup_{m=n}^{\infty}\{\|\hat S_m -\bu\|<\eps \} \right)\\
& = \lim_{n\to \infty} \Pr \left(\bigcup_{m=n}^{\infty}\{\|\hat S_m -\bu \|<\eps\} \right)\\
& \geq \lim_{n\to \infty} \Pr(\|\hat S_n-\bu\|<\eps)\\
&= \Pr ( \| \zeta - \bu \| < \eps ),
\end{align*}
which is strictly positive provided $\bu\in\supp\zeta$. 
It follows by the Hewitt--Savage theorem that if $\bu\in\supp\zeta$, then
$\Pr(\|\hat S_n-\bu\|<\eps \io)=1$ for all $\eps >0$, and hence $\bu \in \cD$.

For part (ii), we have that since
  $\Pr (\xi = \0 ) = 0$, and the function $\bx \mapsto \hat \bx$ is continuous
on $\R^d \setminus \{ \0 \}$, the continuous mapping theorem
implies that
$\hat S_n \tod \hat \xi$, and then we may apply part~(i).
\end{proof}

Here is a sufficient condition for $\cD  = \Sp{d-1}$;
if $d=2$ the walk is recurrent and the result  also follows from~Proposition~\ref{prop:recurrence},
while if $d \geq 3$ the walk is transient.

\begin{corollary}
\label{cor:two-moments}
Suppose that $\Exp ( \| X \|^2) < \infty$ and $\mu = \0$. Then $\cD = \Sp{d-1}$.
\end{corollary}
\begin{proof}
By assumption and the central limit theorem, $n^{-1/2} S_n$ converges in distribution to a non-degenerate
normal distribution. Proposition~\ref{prop:D-supp} then shows that $\Sp{d-1} \subseteq \cD$.
\end{proof}

\section{Compactification and growth rates}
\label{sec:comments}

Let $\bR$ denote the compactification of $\R^d$
obtained by adjoining the ``sphere at $\infty$''.
More formally, $\bR$ is the compact metric space
obtained by the completion of $\R^d$ with respect
to the metric
\[ \rho(\bx,\by) = \left\| \frac{\bx}{1 + \| \bx \|} - \frac{\by}{1+\|\by\|} \right\| .\]
Then we can represent $\bR$ as $\bR = \R^d \cup \R_\infty^d$
where $\R_\infty^d$ is in bijection to $\Sp{d-1}$.
We write elements of $\R_\infty^d$ as $\infty \cdot \bu$ for $\bu \in \Sp{d-1}$.
The metric $\rho$ on $\R^d$ is equivalent to the Euclidean metric, and extended
to $\bR$ it is such that $\bx_n \in \R^d$ has $\bx_n \to \infty \cdot \bu$ for $\bu \in \Sp{d-1}$
if $\| \bx_n \| \to \infty$ and $\hat \bx_n \to \bu$.

The set of accumulation points of $S_0, S_1, S_2, \ldots$,
taken in $\bR$, thus consists of any accumulation points in $\R^d$
(a.s.~there are none if $S_n$ is transient)
and accumulation points in $\R_\infty^d$ represented by the set
$\cL_\infty$~of asymptotic directions, as defined at~\eqref{eq:asymptotic-directions}.

Erickson~\cite{erickson76}, generalizing one-dimensional work of Kesten and himself~\cite{kesten70,ek},
considers a finer graduation of asymptotic directions. For $\alpha \in \RP$, set
\[ \cL^{>\alpha}_\infty := \left\{ \bu \in \Sp{d-1} : \liminf_{n \to \infty} \left( \frac{n^\alpha}{1+\| S_n \|} + \| \hat S_n - \bu \| \right) = 0  \right\} .\]
Then $\cL^{>0}_\infty = \cL_\infty$, while $\cL^{>\alpha_2}_\infty \subseteq \cL^{>\alpha_1}_\infty$
for any $0 \leq \alpha_1 \leq \alpha_2 < \infty$. Similarly, set
\[ \cD^{>\alpha}_\infty  := \left\{ \bu \in \Sp{d-1} : \liminf_{n \to \infty} \left( \frac{n^\alpha}{1+\| S_n \|} + \| \hat S_n - \bu \| \right) = 0 , \as \right\} .\]
Roughly speaking, the set $\cL^{>\alpha}_\infty$ consists of those directions in which the walk
grows at rate faster than $n^\alpha$. Also for $\alpha > 0$ set 
\begin{equation}
\label{eq:A-alpha}
\cA^\alpha  = \left\{ \bx \in \R^d : \liminf_{n \to \infty} \left\| n^{-\alpha} S_n - \bx \right\| = 0 \right\}, \end{equation}
and $\cL^\alpha_ \infty = \{ \hat \bx : \bx \in \cA^\alpha \setminus \{ \0 \} \}$.
Then $\cL^\alpha_\infty \subseteq \cL_\infty$
are those asymptotic directions in which the walk grows at rate precisely~$n^\alpha$.

Erickson~\cite{erickson76,erickson00} studies in detail~$\cA^\alpha$ and $\cL^{>\alpha}_\infty$,
with particular focus on the case~$\alpha=1$, which has some peculiar features. The version of Theorem~\ref{thm:D} stated by Erickson~\cite[p.~802]{erickson76} is that $\Pr ( \cL^{>\alpha}_\infty = \cD^{>\alpha}_\infty ) = 1$,
and $\cD^{>\alpha}_\infty$ is a closed subset of~$\Sp{d-1}$. 

For $d \geq 3$, the value $\alpha = 1/2$ is special,
since a remarkable paper of Kesten~\cite{kesten78} shows that $n^{-\alpha} \| S_n \| \to \infty$ for any $\alpha < 1/2$ and any genuinely $d$-dimensional 
random walk~$S_n$ in $\R^d$, $d \geq 3$. Thus for $d \geq 3$
we have $\cD^{>\alpha}_\infty = \cD_\infty$ for any $0 \leq \alpha < 1/2$.

\section{Limiting direction}
\label{sec:direction}

By the Hewitt--Savage theorem, 
$\Pr ( \lim_{n \to \infty} \hat S_n \text{ exists} ) \in \{ 0,1\}$,
and if the limit exists, then
it is a.s.~constant. 
If $\lim_{n\to\infty} \| S_n \| = \infty$ a.s.~and $\lim_{n \to \infty} \hat S_n = \bu$ a.s.~for some $\bu \in \Sp{d-1}$, 
we say that
 $S_n$ is  \emph{transient with limiting direction} $\bu$.

\begin{lemma}
\label{lem:lim-dir}
Let $\bu \in \Sp{d-1}$.
The following are equivalent.
\begin{itemize}
\item[(i)] $\cD = \{ \bu \}$.
\item[(ii)] $\lim_{n \to \infty} \hat S_n = \bu$, a.s.
\item[(iii)] $S_n$ is transient with limiting direction $\bu$.
\end{itemize}
\end{lemma}
\begin{proof}
The result will follow from the sequence of implications (iii) $\Rightarrow$ (ii) $\Rightarrow$ (i) $\Rightarrow$ (iii).
That (iii) implies (ii) is trivial.
If (ii) holds, then clearly $\bu \in \cD$, and for any $r >0$ we have $\hat S_n \in B_s ( \bu ; r)$ for all but finitely many $n$. 
For any $\bv \in \Sp{d-1} \setminus \{ \bu \}$, we may choose $r>0$ sufficiently small
so that $B_s (\bu ; r)$ and $B_s (\bv; r)$ are disjoint,  so
that $\Pr ( \hat S_n \in B_s ( \bv ; r) \io ) = 0$,
and hence Proposition~\ref{prop:D-cones} shows that $\bv \notin \cD$. Thus (i)~holds.

Finally, suppose that (i) holds. Then Corollary~\ref{cor:not-full-sphere}
shows that $S_n$ is transient, and in particular $S_n = \0$ only finitely often. 
By the Hewitt--Savage theorem, $\limsup_{n \to \infty} \| \hat S_n - \bu \|$ is a.s.~constant.
If $\bu$ is not a limiting direction for the walk,
then  this constant is strictly positive, so that, for some $\eps>0$, $\| \hat S_n - \bu \| \geq \eps$ i.o., a.s. 
Since the set $\{ \bv \in \Sp{d-1} : \| \bv - \bu \| \geq \eps \}$
is compact, it follows that $\hat S_n$ has an accumulation point $\bv \neq \bu$, and hence $\bv \in \cD$, 
which gives
a contradiction. Hence (i) implies (iii).
\end{proof}

The following result is contained in Theorem~1.6.1(i) of~\cite{mpw}.

\begin{proposition}
\label{prop3}
Suppose that $\Exp \| X \| < \infty$. If $\mu \neq \0$, then $\cD = \{ \hat \mu \}$.
\end{proposition}

\begin{remark}
If $\mu = \0$ there is no limiting direction: see Proposition~\ref{prop:zero-drift} below.
\end{remark}

\begin{proof}[Proof of Proposition~\ref{prop3}.]
The strong law of large numbers (SLLN) shows that $n^{-1} S_n \to \mu$, a.s.,
and $n^{-1} \| S_n  \| \to \| \mu \|$, a.s. If $\mu \neq \0$, then $\| S_n \| \to \infty$, so 
$S_n \neq \0$ for all but finitely many $n$, and then
\[ \lim_{n \to \infty} \hat S_n = \lim_{n \to \infty} \frac{n^{-1} S_n}{n^{-1} \| S_n \|} = \hat \mu , \as \qedhere \]
\end{proof}

\section{The zero-drift case}
\label{sec:plane}

In this section we turn to the case where the walk has zero drift, i.e., $\mu = \0$. 
If $d=1$, then zero drift implies recurrence, and hence $\cD = \{ -1,+1\}$
(see e.g.~\cite[Theorem~4.2.7]{dur}). 
If $\Exp ( \| X \|^2 ) < \infty$, then Corollary~\ref{cor:two-moments}
shows that $\cD = \Sp{d-1}$. 
Thus the most interesting cases are when $d \geq 2$ and $\Exp ( \| X \|^2 ) = \infty$.
The following result contrasts with Proposition~\ref{prop3},
and improves on Theorem~1.6.1(ii) of~\cite{mpw}.

\begin{proposition}
\label{prop:zero-drift}
Suppose that $d \geq 2$, $\Exp \| X \| < \infty$, and $\mu = \0$.
Then $\cD$ is uncountable.
\end{proposition}
 
In the case where $d=2$, we can say more.
For measurable $A \subseteq \Sp{d-1}$ we write $|A|$ for the Haar measure of $A$.
Write `$\eqd$' for equality in distribution;
$X \eqd -X$ means that random variable $X \in \R^d$ has a centrally symmetric distribution.

\begin{proposition}
\label{prop:plane}
Suppose that $d=2$, $\Exp \| X \| < \infty$, and $\mu = \0$.
\begin{itemize}
\item[(i)] We have $| \cD | \geq \frac{1}{2} | \Sp{1} |$.
\item[(ii)] If $X \eqd -X$, then $\cD = \Sp{1}$.
\end{itemize}
\end{proposition}

\begin{remarks}
\begin{myenumi}
\setlength{\itemsep}{0pt plus 1pt}
\item[{\rm (a)}] Example~\ref{ex:ber-alpha} below gives a walk with $d=2$, $X \eqd -X$, and $\Exp \| X \| = \infty$, for which $\cD$
has only two elements, so the condition $\Exp \| X \| < \infty$ in Proposition~\ref{prop:plane} cannot be  removed.
 \item[{\rm (b)}]  Example~\ref{ex:four-dim} below gives a family of random walks in $\R^d$, $d \geq 4$,
for which $\mu = \0$ and $X \eqd -X$, but $\cD$ is a set of measure zero,
so in higher dimensions the hypotheses of Proposition~\ref{prop:plane} do not guarantee that $\cD$ occupies a positive fraction of the sphere.
 \end{myenumi}
\end{remarks}

For further results in the zero-drift case, see Corollary~\ref{cor:zero-drift-hull} below.
In the rest of this section we prove Propositions~\ref{prop:zero-drift} and~\ref{prop:plane}.

\begin{lemma}
\label{lem:orthogonal-span}
Suppose that $d \geq 2$, $\Exp \| X \| < \infty$, and $\mu = \0$.
Then for every $\bu \in \Sp{d-1}$, there exists $\bv \in \cD$ with $\bu \cdot \bv = 0$.
\end{lemma}
\begin{proof}
If $S_n$ is recurrent, then the result follows from Proposition~\ref{prop:recurrence}.
So suppose that $S_n$ is transient. Fix $\bu \in \Sp{d-1}$. For $\eps >0$,
let $O_\eps (\bu ) = \{ \bv \in \Sp{d-1} : | \bv \cdot \bu | \leq \eps \}$.
Since $\Exp ( X \cdot \bu ) = \mu \cdot \bu = 0$,
the random walk $S_n \cdot \bu$ is recurrent, and $\liminf_{n \to \infty} | S_n \cdot \bu | < \infty$. 
Since $S_n$ is transient we have $\| S_n \| \to \infty$, so that $\liminf_{n \to \infty} | \hat S_n \cdot \bu | = 0$.
In other words, for every $\eps >0$ we have that for infinitely many $n \in \N$, $\hat S_n$
is in the compact set $O_\eps (\bu)$. Hence $O_\eps (\bu)$ must contain an element of $\cD$.
Thus there is a sequence $\bv_1,\bv_2, \ldots \in \cD$ with $| \bv_j \cdot \bu | \to 0$,
and (since $\cD$ is compact) this sequence has a subsequence which converges to $\bv \in \cD$
with $\bv \cdot \bu = 0$.
\end{proof}

\begin{proof}[Proof of Proposition~\ref{prop:zero-drift}.]
Suppose, for the purpose of deriving a contradiction, that $\cD$ is countable.
Set $O  (\bu ) = \{ \bv \in \Sp{d-1} :  \bv \cdot \bu = 0 \}$.
Then $O = \cup_{\bu \in \cD} O (\bu)$ is a countable union
of subsets of $\Sp{d-1}$ of measure zero (since each $O(\bu)$
is a copy of $\Sp{d-2}$).
Thus $O$ is measure zero,
and so there exists $\bv \in \Sp{d-1} \setminus O$.
This $\bv$ has $\bv \cdot \bu \neq 0$ for all $\bu \in \cD$,
which contradicts Lemma~\ref{lem:orthogonal-span}.
Hence $\cD$ cannot be countable.
\end{proof}

To prove Proposition~\ref{prop:plane}, we need some additional notation.
Let
\begin{align*} \cD_1 & := \{ \bu \in \Sp{d-1} : \bu \in \cD, ~ - \bu \notin \cD \}, \\
\cD_2 & := \{ \bu \in \Sp{d-1} : \bu \in \cD, ~ - \bu \in \cD \}, \\
\cC_1 & := \{ \bu \in \Sp{d-1} : \bu \notin \cD, ~ - \bu \in \cD \} = - \cD_1 , \\
\cC_2 & := \{ \bu \in \Sp{d-1} : \bu \notin \cD, ~ - \bu \notin \cD \} .\end{align*}
Then $\cD = \cD_1 \cup \cD_2$ and $\Sp{d-1} \setminus \cD = \cC_1 \cup \cC_2$.

\begin{lemma}
\label{lem:antipodes}
Suppose that $d=2$, $\Exp \| X \| < \infty$, and $\mu = \0$. Then $\cC_2 = \emptyset$.
\end{lemma}
\begin{proof}
Lemma~\ref{lem:orthogonal-span} shows that for every $\bu \in \Sp{1}$, there exists $\bv \in \Sp{1}$ 
such that $\bu \cdot \bv = 0$ and $\bv \in \cD$.
As $\bu$ runs over $\Sp{1}$, the set of $\pm \bv$ such that $\bu \cdot \bv = 0$
runs over the whole of $\Sp{1}$, and so in this case we conclude that
for every $\bu \in \Sp{d-1}$, at least one of $\pm \bu$ is in $\cD$.
Hence $\cC_2 = \emptyset$.
\end{proof}

\begin{proof}[Proof of Proposition~\ref{prop:plane}.]
Note that $| \cD | = | \cD_1 | + | \cD_2 |$.
If  Lemma~\ref{lem:antipodes} applies, then we have $| \Sp{1} \setminus \cD |  =  | \cC_1 | = | \cD_1 |$.
Hence $| \Sp{1} | = 2 | \cD_1 | + | \cD_2 |$, and part~(i) follows.
If $X \eqd -X$, then $\cD = - \cD$, so $\cD_1 = \cC_1 = \emptyset$. Thus $\Sp{d-1} = \cD_2 \cup \cC_2$.
If  Lemma~\ref{lem:antipodes} applies, then $\cD = \cD_2 = \Sp{1}$, giving part~(ii).
\end{proof}

\section{An arbitrary set of recurrent directions}
\label{sec:erickson}

We know from Theorem~\ref{thm:D} that the set $\cD$ is closed. The aim of this section is to show that
there are, in general, no other restrictions on $\cD$: it can be an arbitrary closed subset of the sphere.
This result is essentially due to Erickson~\cite[pp.~508--510]{erickson00}; we reproduce the argument here.

\begin{theorem}
\label{thm:heavytails}
Let $A$ be a non-empty closed subset of $\Sp{d-1}$. Suppose that the increment distribution of the random walk is given by
 $X =Q \xi$ where $Q \in \Sp{d-1}$ and $\xi \in\RP$ are independent, $\Pr ( \xi > 0 ) >0$, and $\supp Q = A$.
Let $\xi_1, \xi_2, \ldots$ be independent copies of $\xi$, and suppose that
\begin{equation}
\label{eq:max-dominates}
\lim_{n \to \infty} \frac{\max_{1 \leq i \leq n} \xi_i}{\sum_{i=1}^n \xi_i} = 1, \as 
\end{equation}
Then the recurrent directions of the random walk $S_n=\sum_{i=1}^{n}X_i$ are $\cD=A$.
\end{theorem}
\begin{remarks}
\begin{myenumi}
\setlength{\itemsep}{0pt plus 1pt}
\item[{\rm (a)}] 
 Pruitt, in Theorem~2 of~\cite{pruitt}, shows that~\eqref{eq:max-dominates} 
holds if and only if $\sum_{k \geq 1} u_k^2 < \infty$, where $u_k = \Pr ( 2^k < \xi \leq 2^{k+1} ) / \Pr (2^k < \xi)$.
Examples that work have very heavy tails, and include
$\Pr ( \xi > r ) = 1/\log r$ for $r \geq \re$ (see \cite[pp.~509--510]{erickson00})
and $\Pr ( \xi > r ) = \exp ( - (\log r)^\beta )$ for $r \geq 1$ with $\beta \in (0,1/2)$
(see \cite[p.~895]{pruitt}).
\item[{\rm (b)}] 
The intuition behind Theorem~\ref{thm:heavytails} is as follows. The condition~\eqref{eq:max-dominates} means that the biggest jump so far is a.s.~on a bigger scale than all the other
jumps combined, and so the projection on the sphere is determined by the $Q$ corresponding to the current biggest jump. As times goes on,
one sees an i.i.d.~subsequence of the $Q$s associated with the biggest jumps, and so the walk explores the sphere over the set $A$. 
\item[{\rm (c)}] 
Theorem~\ref{thm:heavytails} can be compared to the construction of random walks
with desired limit properties of~\cite{kesten70,ek,erickson76,erickson00}.
\end{myenumi}
\end{remarks}

\begin{proof}[Proof of Theorem~\ref{thm:heavytails}.]
Write $X_i = Q_i \xi_i$ where the $Q_i$ are i.i.d.~copies of $Q$ and the $\xi_i$ are i.i.d.~copies of $\xi$.
Let $T_n = \sum_{i=1}^n \xi_i$, $M_n = \max_{1 \leq i \leq n} \xi_i$, and $B_n = T_n - M_n$; then~\eqref{eq:max-dominates}
is equivalent to $B_n/M_n \to 0$, a.s.  
Also set $k(1):=1$ and, for $n \in \N$,
\[ k(n+1) := \begin{cases} k(n) & \text{if } \xi_{n+1} \leq M_n,\\
n+1 & \text{if } \xi_{n+1} > M_n . \end{cases} \]
Then $M_n = \xi_{k(n)}$. Define $R_n := S_n - M_n Q_{k(n)}$.
Since $\| Q_{k(n)} \| = 1$, repeated application of the triangle inequality yields
\begin{align*}
\| \hat S_n - Q_{k(n)} \| & = \left\| \frac{M_n Q_{k(n)} + R_n - \| S_n\| Q_{k(n)}}{\| S_n \|} \right\| \\
& \leq \frac{| M_n - \| S_n \| | }{\| S_n \|} +  \frac{\| R_n \|}{\| S_n \|} \\
& \leq \frac{2 \| R_n \|}{M_n - \| R_n \|} .\end{align*}
But $\| R_n \| =  \| \sum_{i \in \{1,\ldots , n\} \setminus \{ k(n)\}} X_i \| \leq B_n$ where $B_n = T_n - M_n$, so
\[ \| \hat S_n - Q_{k(n)} \|  \leq \frac{2 (B_n/M_n)}{1 - (B_n/M_n)} \to 0 , \as, \]
by~\eqref{eq:max-dominates}.

Since $M_n$ is a non-decreasing sequence in $\RP$ with $M_n \to \infty$ a.s.~(as easily follows from~\eqref{eq:max-dominates} and the fact that $\Pr ( \xi > 0 ) >0$) the sequence $k(1), k(2), \ldots$
is a non-decreasing subsequence of $\ZP$ with $k(n) \to \infty$ a.s., and since the $Q_i$ are independent of the $\xi_i$, the sequence
$k(1), k(2), \ldots$ is independent of the sequence $Q_1, Q_2, \ldots$. 
Let $\ell_1 = 1$ and for $n \in \N$ define $\ell_{n+1} = \min \{ m > \ell_n : k(m) > k(\ell_n) \}$,
so that $1 = k(\ell_1) < k(\ell_2) < k(\ell_3) < \cdots$.
Then the sequence
$Q_{k(\ell_1)}, Q_{k(\ell_2)}, \ldots$
has the same law as a sequence of i.i.d.~copies of $Q$. Hence if $\bu \in A$ we have
\[ \liminf_{n \to \infty} \| \hat S_n - \bu\| \leq \lim_{n \to \infty}  \| \hat S_n - Q_{k(\ell_n)} \| + \liminf_{n \to \infty} \| Q_{k(\ell_n)} - \bu\| =0, \as \]
Thus $\bu \in \cD$.	This shows that $A \subseteq \cD$.

On the other hand, if $\bu \notin A$ we have that since $\Sp{d-1} \setminus A$ is open in $\Sp{d-1}$ there is some $r >0$ such that $\Pr ( Q \in B_s (\bu ; r ) ) = 0$,
and 
\[ \liminf_{n \to \infty} \| \hat S_n - \bu\| \geq  \liminf_{n \to \infty} \| Q_{k(\ell_n)} - \bu\| - \lim_{n \to \infty}  \| \hat S_n - Q_{k(\ell_n)} \| \geq r , \as ,\]
so that $\bu \notin \cD$. Thus $\cD \subseteq A$ and the proof is complete.
\end{proof}

\section{Convexity and an upper bound}
\label{sec:convexity}

We start this section with a straightforward result (Theorem~\ref{thm:support_bound})
that is sometimes useful for giving an upper bound on $\cD$ in terms of the support of $\hat S_n$.
We then present (in Proposition~\ref{prop:s-hull} below)
a simpler description of the upper bound in terms of the distribution of $X$ alone, rather than its convolutions. To do so, we need an
appropriate notion of convexity, which will also be useful in Sections~\ref{sec:projections} and~\ref{sec:hull} below when we look at one-dimensional projections and the convex hull of the walk.

Let $\cX_n = ( \supp \hat S_n ) \setminus \{ \0 \}$, and let
$\cX^\star = \cl ( \cup_{n \geq 1} \cX_n )$. Here is the upper bound.

\begin{theorem}
\label{thm:support_bound}
We have that  $\cD \subseteq \cX^\star$.
\end{theorem}
\begin{proof}
Suppose that $\bu \in \Sp{d-1} \setminus \cX^\star$.
Since $\cX^\star$ is closed, there exists $r>0$ such that $B_s ( \bu ; r) \cap \cX_n = \emptyset$ for all $n \in \N$,
and so $\Pr ( \hat S_n \in B_s ( \bu ; r ) ) = 0$ for all $n \in \N$. Then the Borel--Cantelli lemma shows that 
$\Pr ( A (\bu ; r ) ) = \Pr ( \hat S_n \in B_s ( \bu ; r) \io ) = 0$. Hence, by Proposition~\ref{prop:D-cones}, we have $\bu \notin \cD$.
Hence $\cD \subseteq \cX^\star$.
\end{proof}

For $\bu, \bv \in \Sp{d-1}$ and $\alpha \in [0,1]$, let
\[ I_\alpha ( \bu, \bv ) := \frac{\alpha \bu + (1-\alpha) \bv}{\| \alpha \bu + (1-\alpha) \bv \| } ,\]
unless $\bu = -\bv$ and $\alpha = 1/2$, in which case we set $I_{1/2} (\bu, -\bu) := \0$.
If $\bu \neq - \bv$, set $I (\bu, \bv) := \{ I_\alpha (\bu,\bv) : \alpha \in [0,1] \}$, and
set $I (\bu, -\bu ) := \{ \bu, -\bu\}$ (i.e., ignore $\alpha = 1/2$).

\begin{definition}
\label{def:s-convex}
Say that $A \subseteq \Sp{d-1}$ is \emph{s-convex} if for every $\bu , \bv \in A$, one has $I (\bu, \bv) \subseteq A$.
\end{definition}

Note that we only need  to check the condition in Definition~\ref{def:s-convex} for $\bv \neq -\bu$.
In words, $A \subseteq \Sp{d-1}$ is s-convex if for any $\bu, \bv \in A$, the radial projection onto $\Sp{d-1}$ of the straight line segment from
$\bu$ to $\bv$ in $\R^d$ lies in $A$. See also Lemma~\ref{lem:s-convex-cone} below.

Denote by $\hull A$ the convex hull of $A \subseteq \R^d$.
For $A \subseteq \Sp{d-1}$, define
\[ \shull A :=\{ \hat \bx : \bx \in \hull A ,\, \bx \neq \0 \} .\]  
We will show (see Lemma~\ref{lem:smallest-s-convex}) that $\shull A$ is s-convex. 
Let $\cX := ( \supp \hat X ) \setminus \{ \0 \}$. 

\begin{proposition}
\label{prop:s-hull}
We have that $\cX^\star = \cl \shull \cX$, and $\cX^\star$ is s-convex.
\end{proposition}

We work towards a proof of Proposition~\ref{prop:s-hull}.
Let $\cX' := \{ \hat \bx : \bx \in   \supp X \}$. 

\begin{lemma}
\label{lem:sphere-support}
For $X \in \R^d$ any random variable,
we have that $\cX = ( \cl \cX' ) \setminus \{ \0 \}$.
\end{lemma}
\begin{proof}
Recall that $\supp X$ is the smallest closed $A \subseteq \R^d$ such that $\Pr ( X \in A ) =1$,
or, equivalently, $\supp X = \{ \bx \in \R^d : \Pr ( X \in B ( \bx ; r) ) >0 \text{ for all } r >0\}$.
Since $\supp \hat X$ is a closed subset of $\Sp{d-1} \cup \{ \0\}$, it follows that $\cX$ is a closed subset of $\Sp{d-1}$.

Suppose that $\bu \in \cX'$ with $\bu \neq \0$. 
Then $\bu r \in \supp X$ for some $r >0$. This means that 
$\Pr ( X \in B ( \bu r ; s ) ) >0$ for all $s \in (0,r/2)$, say;
but, for any $\bx \in B (\bu r ; s)$,
\begin{align*}
 \| \hat \bx - \bu \|
& = \| \bx \|^{-1} \left( \| \bx - \| \bx \|   \bu \| \right) \\
& \leq \| \bx \|^{-1} \left( \| \bx - r \bu \| + | r - \| \bx \| | \right)
\leq 4 s/r ,\end{align*}
so $\Pr ( \hat X \in B ( \bu ; 4 s / r ) ) \geq \Pr ( X \in B ( \bu r ; s )) > 0$ for all $s \in (0,r/2)$. Hence $\bu \in \supp \hat X$.
Thus $\cX'  \subseteq \cX \cup \{ \0 \}$, and since $\cX \cup \{ \0 \}$ is closed we get $\cl \cX' \subseteq \cX \cup \{ \0 \}$.

On the other hand suppose that $\bu \in \cX$. Let $r_n >0$ be such that $r_n \to 0$. Then $\Pr ( X \in C ( \bu ; r_n ) ) = \Pr ( \hat X \in B (\bu; r_n ) )>0$ for all $n$,
which means that $C ( \bu ; r_n ) \cap \supp X \neq \emptyset$, i.e., for every $n$ there exists $\bx_n \in \supp X$ with
$\| \hat \bx_n - \bu \| \leq r_n$. Hence $\hat \bx_n \in \cX'$ with $\hat \bx_n \to \bu$, so $\bu \in \cl \cX'$, and we get $\cX \subseteq \cl \cX'$. 
\end{proof}

The next result characterizes a set as s-convex if and only if all normalized
conical combinations are contained within the set.

\begin{lemma}
\label{lem:s-convex-cone}
The set $A \subseteq \Sp{d-1}$ is s-convex if and only if for all $n \in \N$,
all $\bu_1, \ldots, \bu_n \in A$, and all $\beta_1, \ldots, \beta_n \in (0,\infty)$,
\begin{equation}
\label{eq:cone-closure}
\frac{ \sum_{i=1}^n \beta_i \bu_i }{ \left\|  \sum_{i=1}^n \beta_i \bu_i \right\| } \in A, \text{ whenever } \sum_{i=1}^n \beta_i \bu_i \neq \0.\end{equation}
\end{lemma}
\begin{proof}
The `if' half follows immediately (take $n=2$ and $\beta_1 + \beta_2 = 1$).
Suppose that $A$ is s-convex. We proceed by an induction on~$n$.
Then~\eqref{eq:cone-closure} holds for $n=2$,
since
\[ \frac{\beta_1 \bu_1 + \beta_2 \bu_2}{\| \beta_1 \bu_1 + \beta_2 \bu_2 \|}
= \frac{\frac{\beta_1}{\beta_1+\beta_2} \bu_1 +   \frac{\beta_2}{\beta_1+\beta_2}  \bu_2}
{ \left\| \frac{\beta_1}{\beta_1+\beta_2} \bu_1 +   \frac{\beta_2}{\beta_1+\beta_2}  \bu_2 \right\| } .\]
 Suppose that~\eqref{eq:cone-closure} holds for all $n \in \{1,\ldots,m\}$
with $m \geq 2$,
and consider $\bu_1, \ldots, \bu_{m+1} \in A$ and $\beta_1,\ldots,\beta_{m+1} \in (0,\infty)$
with $\sum_{i=1}^{m+1} \beta_i \bu_i \neq \0$. We may also suppose that $\beta_m \bu_m + \beta_{m+1} \bu_{m+1} \neq \0$,
or else the inductive hypothesis would apply directly. 
Set $\bu_i' = \bu_i$ for $1 \leq i \leq m -1$ and 
\[ \bu'_{m} = \frac{\frac{\beta_{m}}{\beta_{m} +\beta_{m +1}} \bu_{m} + \frac{\beta_{m+1}}{\beta_{m} +\beta_{m +1}} \bu_{m+1}}
{\left\| \frac{\beta_{m}}{\beta_{m} +\beta_{m +1}} \bu_{m} + \frac{\beta_{m+1}}{\beta_{m} +\beta_{m +1}} \bu_{m+1}\right\|} .\]
Then since $A$ is s-convex, $\bu'_m \in A$, and
\[ \frac{ \sum_{i=1}^{m+1} \beta_i \bu_i }{ \left\|  \sum_{i=1}^{m+1} \beta_i \bu_i \right\| } = 
\frac{ \sum_{i=1}^m \beta'_i \bu'_i }{\left\|  \sum_{i=1}^m \beta'_i \bu'_i  \right\|} ,\]
where $\beta'_i = \beta_i$ for $1 \leq i \leq m-1$ and $\beta'_m = \| \beta_m \bu_m + \beta_{m+1} \bu_{m+1} \|$.
By inductive hypothesis, the expression in the last display is thus in~$A$. This completes the
inductive step.
\end{proof}

\begin{corollary}
\label{cor:s-convex-hull}
Suppose that $A \subseteq \Sp{d-1}$ is s-convex. Then $A = \Sp{d-1} \cap \hull A$.
\end{corollary}
\begin{proof}
It is clear that $A \subseteq \Sp{d-1} \cap \hull A$.
So suppose that $\bu \in \Sp{d-1} \cap \hull A$.
Then (see e.g.~Lemma~3.1 of~\cite[p.~42]{gruber}) there exist $n \in \N$,
$\bv_1, \ldots, \bv_n \in A$, and $\lambda_1, \ldots, \lambda_n \in [0,1]$
with $\sum_{i=1}^n \lambda_i = 1$, for which $\bu = \sum_{i=1}^n \lambda_i \bv_i$.
But, since $A$ is s-convex and $\| \bu \| = 1$, Lemma~\ref{lem:s-convex-cone} shows that
$ \sum_{i=1}^n \lambda_i \bv_i \in A$. So $\Sp{d-1} \cap \hull A \subseteq A$.
\end{proof}

The next result shows that $\shull A$ has a similar characterization to the usual $\hull A$.

\begin{lemma}
\label{lem:smallest-s-convex}
 For $A \subseteq \Sp{d-1}$, 
 $\shull A$ is the smallest s-convex $B \subseteq \Sp{d-1}$ with $A \subseteq B$.
\end{lemma}
\begin{proof}
Let $\bu, \bv \in \shull A$ with $\bv \neq -\bu$, and $\alpha \in (0,1)$.
Then $\bu = \hat \bx$ and $\bv = \hat \by$ for some $\bx, \by \in \hull A$ with $\bx , \by \neq \0$.
Choose $\beta \in (0,1)$ given by
\[ \beta = \frac{\alpha \| \by \|}{\alpha \| \by \| + (1-\alpha) \| \bx \|} .\]
Consider $\bw = \beta \bx + (1-\beta) \by$. Then, since $\hull A$ is convex,
$\bw \in \hull A$, and $\bw \neq \0$ since $\hat \bx \neq - \hat \by$, so  $\hat \bw \in \shull A$. But
\[ \frac{\bw}{\| \bw \|} = \frac{\alpha \hat \bx + (1-\alpha) \hat \by}{ \left\| \alpha \hat \bx + (1-\alpha) \hat \by \right\|}  \]
is thus in $\shull A$ for all $\alpha \in (0,1)$, verifying that $\shull A$ is s-convex.

Next we claim that if $A \subseteq \Sp{d-1}$ is s-convex, then $\shull A = A$. Clearly $A \subseteq \shull A$.
So suppose that $A$ is s-convex, and consider $\bu \in \shull A$.
Then $\bu = \hat \bx$ for some $\bx \in \hull A$, $\bx \neq \0$,
and thus (see e.g.~Lemma~3.1 of~\cite[p.~42]{gruber}) there exist $n \in \N$,
$\bv_1, \ldots, \bv_n \in A$, and $\lambda_1, \ldots, \lambda_n \in [0,1]$
with $\sum_{i=1}^n \lambda_i = 1$, for which $\bx = \sum_{i=1}^n \lambda_i \bv_i$.
Then Lemma~\ref{lem:s-convex-cone}
shows that $\hat \bx \in A$. In other words, $\shull A \subseteq A$, as required.

Suppose $B$ is s-convex with $A \subseteq B$; then the preceding paragraph shows that
$\shull A \subseteq \shull B = B$, which completes the proof of the lemma.
\end{proof}

\begin{lemma}
\label{lem:closure}
Let $A \subseteq \Sp{d-1}$ be s-convex. Then $\cl A$ is also s-convex.
\end{lemma}
\begin{proof}
It suffices to suppose $\bu, \bv \in \cl A$ with $\bu \neq - \bv$. Then there exist $\bu_1, \bu_2, \ldots \in A$
and $\bv_1, \bv_2, \ldots \in A$ with $\bu_n \to \bu$ and $\bv_n \to \bv$, and there exists $n_0 \in \N$
such that $\bu_n \neq -\bv_n$ for all $n \geq n_0$.
Since $A$ is s-convex, $I_\alpha (\bu_n, \bv_n) \in A$  for all $n \geq n_0$ and all $\alpha \in [0,1]$.
By continuity of the function $\bx \mapsto \hat \bx$ on $\R^d \setminus \{ \0 \}$,
it follows that 
$I_\alpha ( \bu, \bv) = \lim_{n \to \infty} I_\alpha (\bu_n, \bv_n)  \in \cl A$ for all $\alpha \in [0,1]$.
Hence $\cl A$ is s-convex.
\end{proof}

\begin{proof}[Proof of Proposition~\ref{prop:s-hull}.]
First we use induction to show that $\cX_n \subseteq \cl \shull \cX$ for all $n \in \N$.
Clearly this is true for $n=1$. So suppose, for the inductive hypothesis,
that $\cX_m \subseteq \cl \shull \cX$ for all $m \in \{1,\ldots,n\}$.
Now, provided that $S_{n+1} \neq \0$, we have
\begin{align*}
 \hat S_{n+1} &  
 = \frac{ \alpha_n \hat S_n + (1-\alpha_n) \hat X_{n+1} }{ \| \alpha_n \hat S_n + (1-\alpha_n) \hat X_{n+1} \| } ,
\text{ where } \alpha_n = \frac{ \| S_n\|}{\| S_n \| + \| X_{n+1} \|}.
\end{align*}
In particular, since $\Pr ( \hat S_n \in \cX_n \cup \{ \0 \} ) =1$ and $\Pr ( \hat X_{n+1} \in \cX \cup \{ \0 \} ) = 1$, we have 
\[ \Pr \left( \hat S_{n+1}   \in \left( \cup \{ I  ( \bu , \bv) :   \bu, \bv \in \cl \shull \cX \} \right) \cup \{ \0 \} \right) =1 , \]
by the inductive hypothesis. But $\cl \shull \cX$ is s-convex, by Lemmas~\ref{lem:smallest-s-convex} and~\ref{lem:closure}, so $\Pr ( \hat S_{n+1} \in ( \cl \shull \cX ) \cup \{ \0 \} ) =1$,
which means that $\cX_{n+1} \subseteq \cl \shull \cX$, completing the induction.
Thus we conclude that $\cX^\star \subseteq \cl \shull \cX$.

Next we show that $\cX^\star$ is s-convex. It suffices to suppose that $\bu, \bv \in \cX^\star$ with $\bu \neq - \bv$. Then there exist 
sequences $\bu_{n_k} \in \cX_{n_k}$ and $\bv_{m_k} \in \cX_{m_k}$
with $\bu_{n_k} \to \bu$ and $\bv_{m_k} \to \bv$.
Lemma~\ref{lem:sphere-support} shows that, correspondingly, there exist sequences $\bx_{n_{k,1}}, \bx_{n_{k,2}}, \ldots \in \supp S_{n_k}$ and $\by_{m_{k,1}} , \by_{m_{k,2}} ,\ldots \in \supp S_{m_k}$
 with $\lim_{i \to \infty} \hat \bx_{n_{k,i}} = \bu_{n_k}$ and
$\lim_{j \to \infty} \hat \by_{m_{k,j}} = \bv_{m_k}$, and, for all $k$ sufficiently large and all $i,j$ sufficiently large, $\hat \bx_{n_{k,i}} \neq - \hat \by_{m_{k,j}}$.
Now  for $s, t \in \ZP$,
$s \bx_{n_{k,i}} + t \by_{m_{k,j}} \in \supp S_{s n_k+ t m_k}$.
Applying Lemma~\ref{lem:sphere-support} with $X = S_{s n_k+t m_k}$ we see that $\bw \in \cX_{s n_k+ tm_k} \subseteq \cX^\star$,
where
\[ \bw = \frac{s \bx_{n_{k,i}} + t \by_{m_{k,j}}}{\| s \bx_{n_{k,i}} + t \by_{m_{k,j}} \|}
=  I_{\alpha_{s,t,i,j}} ( \hat \bx_{n_{k,i}} , \hat \by_{m_{k,j}} ) ,
 \]
with
\[ \alpha_{s,t,i,j} = \frac{s \| \bx_{n_{k,i}} \|} {s \| \bx_{n_{k,i}} \| + t \| \by_{m_{k,j}} \|} .\]
For fixed $k, i,j$ and $\alpha \in [0,1]$, we may choose $s, t \to \infty$ such that $\alpha_{s,t,i,j} \to \alpha$,
and since for $\bu \neq -\bv$, $\alpha \mapsto I_\alpha (\bu , \bv)$ is continuous over $\alpha \in [0,1]$, and $\cX^\star$ is closed, we get
\[ I_{\alpha} ( \hat \bx_{n_{k,i}} , \hat \by_{m_{k,j}} ) = \lim_{s,t \to \infty} I_{\alpha_{s,t,i,j}} ( \hat \bx_{n_{k,i}} , \hat \by_{m_{k,j}} ) \in \cX^\star, \text{ for all }
\alpha \in [0,1]  .\]
Then by continuity of $(\bu, \bv ) \mapsto I_\alpha (\bu, \bv)$ away from $\bu = -\bv$ we get
\[ I_\alpha (\bu , \bv) = \lim_{k \to \infty} I_\alpha ( \bu_{n_k} , \bv_{m_k} ) =
\lim_{k \to \infty}  \lim_{i, j \to \infty} I_{\alpha} ( \hat \bx_{n_{k,i}} , \hat \by_{m_{k,j}} ) \in \cX^\star ,\]
for all $\alpha \in [0,1]$.
Hence $\cX^\star$ is s-convex, and $\cX \subseteq \cX^\star$, so, by Lemma~\ref{lem:smallest-s-convex},
we have $\shull \cX \subseteq \cX^\star$, and since $\cX^\star$ is closed,
we get $\cl \shull \cX \subseteq \cX^\star$.

Thus we conclude that $\cX^\star = \cl \shull \cX$, and the latter is s-convex by Lemmas~\ref{lem:smallest-s-convex}
and~\ref{lem:closure}.
\end{proof}

We finish this section with a result on the boundary of an s-convex set, which will be useful in Section~\ref{sec:projections} below.
For $A \subseteq \Sp{d-1}$, denote by $\sint A$ the interior of $A$ relative to $\Sp{d-1}$, i.e., $\bu \in \sint A$ if and only if
$B_s ( \bu ; \delta ) \subseteq A$ for some $\delta >0$. Also, for $A \subseteq \Sp{d-1}$, we write $\partial_s A$ for the boundary of $A$ relative to $\Sp{d-1}$,
i.e., $\partial_s A := (\cl A) \setminus (\sint A)$.

\begin{lemma}
\label{lem:boundary}
If $A \subseteq \Sp{d-1}$ is s-convex, then (i) $\sint A = \sint \cl A$; and (ii) $\partial_s A = \partial_s \cl A$.
\end{lemma}
\begin{proof}
Suppose that $\bu \in \sint \cl A$. Then there exist $m \in \N$
and $\bu_1, \ldots, \bu_m \in \cl A$ such that
$\bu \in \sint P_s ( \bu_1, \ldots, \bu_m )$,
where
$P_s ( \bu_1, \ldots, \bu_m ) := \shull \{ \bu_1, \ldots, \bu_m \}$.
Let 
\[ R_s (\bv_1, \ldots, \bv_m ; \bu ) = \inf \{ \| \bv - \bu \| : \bv \in \Sp{d-1} \setminus P_s (\bv_1, \ldots, \bv_m ) \} ,\]
which is zero unless $\bu$ lies in the interior of $P_s (\bv_1, \ldots, \bv_m )$,
when it is equal to the shortest distance from $\bu$ to the boundary of $P_s (\bv_1, \ldots, \bv_m )$.
In particular, note that $R_s (\bu_1, \ldots, \bu_m ; \bu) = \delta_0 > 0$.
For $\bv_1, \ldots, \bv_m \in \Sp{d-1}$,
the map $(\bv_1, \ldots, \bv_m ) \mapsto P_s (\bv_1, \ldots, \bv_m )$, as a 
function from $(\Sp{d-1})^{m}$ to compact subsets of $\R^d$ with the Hausdorff metric, is
continuous.
So
the map from $(\bv_1, \ldots, \bv_m )$ to $R_s (\bv_1, \ldots, \bv_m ; \bu)$ is also continuous. Hence for any $\delta \in (0,\delta_0)$, 
we can find $\eps >0$ sufficiently small such that $B_s ( \bu ; \delta)$ is contained in
$P_s ( \bv_1, \ldots, \bv_m )$ for all $\bv_i \in \Sp{d-1}$ with $\| \bv_i - \bu_i \| < \eps$. 
Since $\bu_i \in   \cl A$, we can find $\bv_i \in A$ with $\| \bv_i - \bu_i \| < \eps$,
which means that $B_s ( \bu ; \delta) \subseteq P_s ( \bv_1, \ldots, \bv_m ) \subseteq A$,
since $A$ is s-convex. Hence $\bu \in \sint A$. This establishes~(i). Then~(ii) follows
since $\partial_s \cl A = \cl A \setminus \sint \cl A = \cl A \setminus \sint  A = \partial_s A$.
\end{proof}

\section{Projection asymptotics}
\label{sec:projections}

In Section~\ref{sec:hull} we study the way in which the random walk fills space via the convex hull of the trajectory.
Pertinent for this is the behaviour of one-dimensional projections of the walk, so we turn to this first. 
For fixed $\bu \in \Sp{d-1}$, the projection $S_n \cdot \bu$ defines a random walk on $\R$, with increment distribution $X \cdot \bu$,
which either tends to $+\infty$, to $-\infty$, or oscillates (see Lemma~\ref{lem:projections-all} below).
However, this, by itself, does not exclude that there might exist (random) $\bu \in \Sp{d-1}$
for which $S_n \cdot \bu$ does something out of the ordinary, such as having a finite $\limsup$.
While not central for what follows, we show that such \emph{exceptional projections} do not exist, at least for $d \leq 2$.

Define the random sets
\begin{align*} 
\cP_+ & := \bigl\{ \bu \in \Sp{d-1} : \lim_{n \to \infty} (S_n \cdot \bu) = +\infty \bigr\} , ~~  \cP_-  := \bigl\{ \bu \in \Sp{d-1} : \lim_{n \to \infty} (S_n \cdot \bu) = -\infty \bigr\} , \\
\cP_\pm & := \bigl\{ \bu \in \Sp{d-1} : -\infty = \liminf_{n \to \infty} (S_n \cdot \bu)
<  \limsup_{n \to \infty} (S_n \cdot \bu) = + \infty \bigr\},
\end{align*}
and their non-random counterparts
\begin{align*} 
\cD_+ & := \bigl\{ \bu \in \Sp{d-1} : \lim_{n \to \infty} (S_n \cdot \bu) = +\infty, \as \bigr\} ,\\
\cD_- & := \bigl\{ \bu \in \Sp{d-1} : \lim_{n \to \infty} (S_n \cdot \bu) = -\infty, \as \bigr\} , \\
\cD_\pm & := \bigl\{ \bu \in \Sp{d-1} : -\infty = \liminf_{n \to \infty} (S_n \cdot \bu)
<  \limsup_{n \to \infty} (S_n \cdot \bu) = + \infty, \as \bigr\},
\end{align*}
Then $\cP_+ = - \cP_-$, $\cP_\pm = - \cP_\pm$, and similarly for the non-random versions.

\begin{lemma}
\label{lem:projections-all}
The sets $\cD_+, \cD_- , \cD_\pm$ partition $\Sp{d-1}$.
\end{lemma}
\begin{proof}
Let $\bu \in \Sp{d-1}$.
Then
(see e.g.~\cite[Theorem~4.1.2]{dur})
  exactly one of the following holds:
(i) $\bu \in \cD_+$, (ii) $\bu \in \cD_-$, (iii) $\bu \in \cD_\pm$, or (iv)
$\Pr ( X \cdot \bu = 0 ) = 1$.
Case (iv) is ruled out by our assumption that the walk is genuinely $d$-dimensional. 
\end{proof}
 
It is not immediately obvious that $\cP_+, \cP_-, \cP_\pm$
also partition $\Sp{d-1}$. 
 We define 
\begin{align*}
\cE_+ & := \bigl\{ \bu \in \Sp{d-1} : \limsup_{n \to \infty} (S_n \cdot \bu ) \in \R \bigr\}, ~~
 \cE_-  := \bigl\{ \bu \in \Sp{d-1} : \liminf_{n \to \infty} (S_n \cdot \bu ) \in \R \bigr\}
.\end{align*}
We call
$\bu \in \cE := \cE_+ \cup \cE_-$ an \emph{exceptional projection}
of the walk. Since $\cE_- = - \cE_+$, we have $\cE = -\cE$. Lemma~\ref{lem:projections-all}
means that $\Pr ( \bu \in \cE ) = 0$
for all fixed $\bu \in \Sp{d-1}$. 
Recall the definition of s-convexity from Definition~\ref{def:s-convex}.

\begin{lemma}
\label{lem:Dplus}
The sets $\cP_+$, $\cP_-$, $\cP_+ \cup \cE_-$, $\cP_- \cup \cE_+$,
$\cD_+$, and $\cD_-$ are s-convex.
\end{lemma}
\begin{proof}
Suppose that $\bu, \bv \in \cP_+$ with $\bv \neq - \bu$. Then
\[ S_n \cdot ( \alpha \bu + (1-\alpha) \bv ) = \alpha S_n \cdot \bu + (1-\alpha ) S_n \cdot \bv ,\]
and both $S_n \cdot \bu$ and $S_n \cdot \bv$ tend to infinity, so $I_\alpha (\bu, \bv) \in \cP_+$
for all $\alpha \in [0,1]$. Hence $\cP_+$ is s-convex, and so is $\cP_- = - \cP_+$ as well.
The argument for $\cD_+$, $\cD_-$ is essentially the same.
Note that $\bu \in \cP_+ \cup \cE_-$
if and only if $\liminf_{n \to \infty} (S_n \cdot \bu) > - \infty$.
Hence if $\bu, \bv \in \cP_+ \cup \cE_-$,
\[ \liminf_{n \to \infty} ( S_n \cdot ( \alpha \bu + (1-\alpha) \bv ) )
\geq \alpha \liminf_{n \to \infty} (S_n \cdot \bu) + (1-\alpha ) 
\liminf_{n \to \infty} (S_n \cdot \bv) > -\infty ,\]
so $\cP_+ \cup \cE_-$ is s-convex; similarly for  $\cP_- \cup \cE_+$.
\end{proof}

The following result shows that random set $\cP_+$ can differ from the non-random set $\cD_+$ in a rather limited way. In particular,
since $\cP_+$ and $\cD_+$ are s-convex (by Lemma~\ref{lem:Dplus}), Proposition~\ref{prop:P-plus-minus}(i) 
with Lemma~\ref{lem:boundary} shows that
$\Pr ( \partial_s \cP_+ = \partial_s \cD_+ ) =1$. Similarly for $\cP_-$ and $\cD_-$.

\begin{proposition}
\phantomsection
\label{prop:P-plus-minus}
\begin{itemize}
\item[(i)] We have  
\begin{align*}
\Pr ( \cl \cP_+ = \cl \cP_+ \cup \cl \cE_- = \cl \cD_+ ) =1, ~\text{and}~ \Pr ( \cl \cP_- = \cl \cP_- \cup \cl \cE_+ = \cl \cD_- ) =1.
\end{align*}
\item[(ii)] Moreover, $\Pr ( \cl \cE_+ \subseteq \partial_s \cD_-) = \Pr ( \cl \cE_- \subseteq \partial_s \cD_+) = 1$.
\end{itemize}
\end{proposition}
\begin{proof}
For part~(i), it suffices to prove the first statement.
For ease of notation, write $\cP = \cP_+ \cup \cE_-$.
Since, by Lemma~\ref{lem:Dplus}, $\cP$ is s-convex, so is $\cl \cP$, by Lemma~\ref{lem:closure}.
Thus, by Corollary~\ref{cor:s-convex-hull}, $\cl \cP = \Sp{d-1} \cap \hull \cl \cP$.
Since $\cl \cP$ is bounded, $A = \hull \cl \cP = \cl \hull \cP$ \cite[p.~45]{gruber}.
The set $A$ is convex and compact,
 and so it is uniquely determined by its support function $h_A : \R^{d} \to \R$ given by $h_A(\bx) = \sup \{ \bx \cdot \by : \by \in A \}$,
which is continuous~\cite[p.~56]{gruber}. Since $\Q^d$ is
dense in $\R^{d}$, $h_A$ is determined by $\{ h_A (\bx) : \bx \in \Q^d \}$.
By the Hewitt--Savage theorem, each member of this countable collection of   random variables is a.s.~constant,
so $h_A$ is a.s.~constant. Thus the set $A$ is non-random, and then 
$\Pr ( \cl \cP =  S ) = 1$ for the non-random closed, s-convex set $S = \Sp{d-1} \cap A$.
Note that 
\[ \Pr ( \bu \in \cl \cP ) = \begin{cases} 1 & \text{if } \bu \in S,\\
0 &\text{if } \bu \notin S. \end{cases} \]

Since every $\bu \in \cD_+$ has $\Pr ( \bu \in \cP_+ \subseteq \cP ) = 1$, we have $\cD_+ \subseteq S$,
and since $S$ is closed, $\cl \cD_+ \subseteq S$.
On the other hand, if $S \setminus \cl \cD_+ \neq \emptyset$,
there is some $\bu \in S \setminus \cl \cD_+$ and some $\eps >0$
such that $S \cap B_s (\bu ; \eps)$ 
does not intersect $\cl \cD_+$.
The compact set $S$  contains a countable dense subset, $Q$, say,
and every $\bv \in Q \cap B_s (\bu ; \eps)$ has $\bv \notin \cD_+$,
 so $\Pr ( \bv \in \cP_+) =0$. Also, $\Pr ( \bv \in \cE_- ) = 0$. 
Thus no member of $Q \cap B_s (\bu ; \eps)$ is in $\cP$. Since $\cP$ is s-convex with closure~$S$, this implies that there is a neighbourhood of $\bu$ in $S$
that does not intersect $\cP$. Hence $\bu \in S \setminus \cl \cP$. But $\Pr ( S \setminus \cl \cP = \emptyset ) = 1$.
Thus $\Pr ( \cl \cP_+ \cup \cl \cE_- =  \cl \cD_+) = 1$. Repeating the preceding argument, but taking $\cP = \cP_+$ throughout,
gives $\Pr ( \cl \cP_+   =  \cl \cD_+) = 1$~too.

For part (ii), we have from~(i) that $\Pr (   \cl \cE_- \subseteq \cl \cP_+ ) =1$. Moreover,  
we must have $\Pr (   \cl \cE_- \cap \,\sint \cP_+ = \emptyset ) =1$, or else we would have $\cE_- \cap \cP_+ \neq \emptyset$. Thus 
$\Pr ( \cl \cE_- \subseteq \partial_s \cP_+ ) = 1$. But since $\cP_+$ and $\cD_+$ are s-convex and a.s.~have the same closure,  
Lemma~\ref{lem:boundary} shows that $\Pr (  \partial_s \cP_+ =   \partial_s \cD_+) = 1$.
This gives~(ii).
\end{proof}

\begin{corollary}
\label{cor:P-plus-minus}
If $\cD_\pm = \Sp{d-1}$, then $\Pr ( \cP_\pm = \Sp{d-1} ) = 1$.
\end{corollary}
\begin{proof}
If $\cD_\pm = \Sp{d-1}$, then $\cD_+ = \cD_- = \emptyset$, by Lemma~\ref{lem:projections-all}, 
and Proposition~\ref{prop:P-plus-minus}
shows that $\Pr ( \cl \cP_+ \cup \cl \cP_- \cup \cl \cE = \emptyset) = 1$.
\end{proof}

We turn briefly to the question of whether $\cE$ is in fact empty.

\begin{lemma}
\label{lem:perfect}
With probability 1, $\cl \cE$ is a perfect set.
\end{lemma}
\begin{proof}
For a measurable $B \subseteq \Sp{d-1}$ and let $N(B) = \# (B \cap \cl \cE)$,
the number of points of $\cl \cE$ in $B$. 
We claim that, for any $B$ that is open in $\Sp{d-1}$, 
\begin{equation}
\label{eq:number-of-projections}
\Pr ( N(B) = 0 ) = 1 \text{ or } \Pr ( N(B) = \infty ) =1.\end{equation}
Indeed, the $\ZP \cup \{ \infty \}$-valued random variable $N(B)$
is a.s.~constant, by the Hewitt--Savage theorem: $\Pr ( N(B) = K ) = 1$
for some (non-random) $K$. If $1 \leq K < \infty$, we may label the
elements of $B \cap \cl \cE = B \cap \cE$ in an arbitrary order as $\bu_1, \ldots, \bu_K$,
and each is a.s.~constant, by the Hewitt--Savage theorem again,
so there exist constant $\bu_1, \ldots, \bu_K \in B$
with $\Pr ( \bu_j \in  \cE ) = 1$ for each $j$. But $\Pr ( \bu \in \cE ) = 0$ for all $\bu$.
 Hence $K \in \{0,\infty\}$. This establishes~\eqref{eq:number-of-projections}.

Recall that $\cR$ denotes the (countable) set of all $B_s (\bu ; r)$ with $\bu \in \Q^d \cap \Sp{d-1}$
and $r \in \Q \cap (0,\infty)$. From~\eqref{eq:number-of-projections} we have that $\Pr ( N(B) \in \{ 0 , \infty \} \text{ for all } B \in \cR ) = 1$, which means
 that $\cl \cE$ contains no isolated points.
\end{proof}

\begin{corollary}
\label{cor:exceptional}
Suppose that $d \in \{1,2\}$. Then $\Pr ( \cl \cE = \emptyset ) =1$.
\end{corollary}
\begin{proof}
For $d=1$ this is evident, so suppose that $d=2$. By Proposition~\ref{prop:P-plus-minus}, $\Pr ( \cl \cE_- \subseteq \partial_s \cD_+ ) = 1$, while
Lemma~\ref{lem:Dplus} shows that $\cD_+$ is s-convex, so $\partial_s \cD_+$
  contains at most two points. 	Similarly for $\cl \cE_+$. Thus $\cl \cE$ has at most four points. Lemma~\ref{lem:perfect} then shows that $\Pr ( \cl \cE = \emptyset ) = 1$.
\end{proof}

\section{The convex hull}
\label{sec:hull}

For $n \in \ZP$ let $\cH_n := \hull \{ S_0, S_1, \ldots, S_n \}$ (a convex polytope).
Set $\cH_\infty := \cup_{n \geq 0} \cH_n$.
If $x, y \in \cH_\infty$ then $x, y \in \cH_n$ for some $n$, and since $\cH_n$
is convex, $\theta x + (1-\theta) y \in \cH_n \subseteq \cH_\infty$ for all $\theta \in [0,1]$.
Thus $\cH_\infty$ is convex, and hence so is $\cl \cH_\infty$~\cite[p.~44]{gruber}.
Define
\begin{equation}
\label{eq:S-infinity}
 \cS_\infty := \{ S_0, S_1, \ldots \} .\end{equation}
If $S_n$ is transient, then $\cS_\infty$ has no finite limit points.
Since $\cH_n \subseteq \hull \cS_\infty$, we have~$\cH_\infty \subseteq
\hull \cS_\infty$, while $\cH_\infty$ is a convex set containing $\cS_\infty$, so
$\hull \cS_\infty \subseteq \cH_\infty$. That is,
\[  \cH_\infty = \hull \cS_\infty = \hull \{ S_0, S_1, S_2, \ldots \} .\]
Also define
\[ r_n := \inf \{ \| \bx \| : \bx \in \R^d \setminus \cH_n \} .\]
Note that $r_n$ is non-decreasing, so $r_\infty := \lim_{n \to \infty} r_n$
exists in $[0,\infty]$.
In~\cite{mcrw} it is shown that if $\Pr ( r_\infty = \infty) =1$, then
there is a zero--one law for random variables that are tail-measurable
for the sequence $\cH_0, \cH_1, \cH_2, \ldots$: see~\cite[\S 3]{mcrw}.

\begin{lemma}
\label{lem:r-infinity}
We have $\Pr ( r_\infty = \infty ) = \Pr ( \cH_\infty = \R^d ) \in \{0,1\}$.
\end{lemma}
\begin{proof}
By definition of $r_n$, 
we have $B(\0; r_n) \subseteq \cH_n \subseteq \cH_\infty$. Thus if $r_\infty = \infty$,
we have $\cH_\infty = \R^d$. On the other hand, if 
$\cH_\infty = \R^d$, then for any $r \in (0,\infty)$
there exists some $n \in \N$ for which $B(\0 ; r ) \subseteq \cH_n$. (If not, there is some $r$
and $\bx \in B(\0;r)$ with $\bx \notin \cH_\infty$.) Then $r_n \geq r$, so $r_\infty \geq r$.
Since $r$ was arbitrary, we get $r_\infty = \infty$. Thus $\Pr ( r_\infty = \infty ) = \Pr ( \cH_\infty = \R^d )$, and the proof is 
completed by the Hewitt--Savage theorem.
\end{proof}

A consequence of a theorem of Carath\'eodory is that if $A \subseteq \R^d$ is compact, then $\hull A$ is also compact (see e.g.~Corollary~3.1
of \cite[p.~44]{gruber}). Thus $\hull \cD$ is compact, by Theorem~\ref{thm:D}.
The following result relates several concepts from earlier  to the question of whether the convex hull eventually fills all of space.
Here `$\Int$' denotes interior.

\begin{theorem}
\label{thm:hull}
Consider the following statements.
\begin{itemize}
\item[(i)] $\0 \in \Int \hull \cD$.
\item[(ii)] $\Pr ( r_\infty = \infty ) = 1$.
\item[(iii)] $\Pr ( \cH_\infty = \R^d ) =1$.
\item[(iv)] $\cD_\pm = \Sp{d-1}$.
\item[(v)] $\0 \in  \hull \cD$.
\end{itemize}
Then the following logical relationships apply: (i) $\Rightarrow$ (ii) $\Leftrightarrow$ (iii) $\Leftrightarrow$ (iv)   $\Rightarrow$ (v).
\end{theorem}

\begin{remarks}
\begin{myenumi}
\setlength{\itemsep}{0pt plus 1pt}
\item[{\rm (a)}] 
If the random walk is recurrent, then $\cD = \Sp{d-1}$ (Proposition~\ref{prop:recurrence})
and so (i) and hence (iv) hold, so that $\cD_+ = \emptyset$. 
In other words, if $\cD_+ \neq \emptyset$, then $\cD \neq \Sp{d-1}$, and the random walk is transient.

\item[{\rm (b)}] 
Examples~\ref{ex:drift-alpha}
and~\ref{ex:ber-alpha} below show that (i) is not necessary for (iii), and (v) is not sufficient for (iii).
\end{myenumi}
\end{remarks}

In~\cite{mcrw}, it was shown that sufficient for $\Pr ( \cH_\infty = \R^d ) =1$
is that the random walk is recurrent; this follows from Theorem~\ref{thm:hull}
and the fact that recurrence implies that $\cD = \Sp{d-1}$ (Proposition~\ref{prop:recurrence}). Here are some further sufficient
conditions.

\begin{corollary}
\label{cor:zero-drift-hull}
Suppose that either (i) $X \eqd -X$, or (ii) $\Exp \| X \| < \infty$ and $\mu = \0$.
Then $\Pr ( \cH_\infty = \R^d ) =1$.
\end{corollary}
\begin{proof}
By Theorem~\ref{thm:hull}, it suffices to 
show that $\cD_\pm = \Sp{d-1}$. But under either  hypotheses (i) or (ii), the non-degenerate one-dimensional
random walk with increment distribution $X \cdot \bu$ oscillates.
\end{proof}

\begin{proof}[Proof of Theorem~\ref{thm:hull}.]
First suppose that (i) holds. 
If $\0 \in \Int \hull \cD$ then
there exist $m \in \N$ and $\bu_1, \ldots , \bu_m \in \cD$
such that $\0$ is also in the interior of the convex polytope $P ( \bu_1, \ldots, \bu_m ) := \hull \{ \bu_1, \ldots, \bu_m \}$. 
Let 
\[ R (\bv_1, \ldots, \bv_m) = \inf \{ \| \bx \| : \bx \in \R^d \setminus P (\bv_1, \ldots, \bv_m ) \} ,\]
which is zero unless $\0$ lies in the interior of $P (\bv_1, \ldots, \bv_m )$,
when it is equal to the shortest distance from $\0$ to the boundary of $P (\bv_1, \ldots, \bv_m )$.
In particular, note that $R(\bu_1, \ldots, \bu_m) = \delta_0 > 0$.

For $\bv_1, \ldots, \bv_m \in \R^d$,
the map $(\bv_1, \ldots, \bv_m ) \mapsto P (\bv_1, \ldots, \bv_m )$, as a 
function from $\R^{md}$ to convex, compact subsets of $\R^d$ with the Hausdorff metric, is
continuous.
So
the map from $(\bv_1, \ldots, \bv_m )$ to $R (\bv_1, \ldots, \bv_m)$ is also continuous. Hence for any $\delta \in (0,\delta_0)$, 
we can find $\eps >0$ sufficiently small such that $B ( \0 ; \delta)$ is contained in
$P ( \bv_1, \ldots, \bv_m )$ for all $\bv_i$ with $\| \bv_i - \bu_i \| < \eps$. 
For such an $\eps>0$, let 
\[ C_i (r,\eps)  = \{ \bx \in \R^d : \| \hat \bx - \bu_i \| < \eps, \, \| \bx \| \geq r \}. \]
Then for any $\bx_1, \ldots, \bx_m$ with $\bx_i \in C_i (r,\eps)$,
we have that $\hull \{ \hat \bx_1, \ldots, \hat \bx_m \}$ contains the ball $B ( \0 ; \delta )$.
Thus,
since $\| \bx_i \| \geq r$,
\[ B (\0 ; r\delta) \subseteq \hull \{ r \hat \bx_1, \ldots, r \hat \bx_m \} \subseteq \hull \{ \bx_1, \ldots, \bx_m\} .\]
Since $\bu_i \in \cD = \cD_\infty$ (by Theorem~\ref{thm:D-infty}), we have $S_n \in C_i (r,\eps)$ i.o., a.s. 
Thus $B ( \0 ; r \delta ) \subseteq \cH_n$
for all but finitely many $n$. That is $\liminf_{n \to \infty} r_n \geq r \delta$, a.s.
Since $r >0$ was arbitrary, we get $r_\infty = \infty$, a.s. Thus (i) implies (ii), and (ii)
is equivalent to (iii) by Lemma~\ref{lem:r-infinity}.

Suppose that $\bu \in \cD_+$, so that $\Pr (\bu \in \cL_+) = 1$. Then $S_n \cdot \bu \to \infty$,
so that $\inf_{n \geq 0} ( S_n \cdot \bu ) = c$
for some $c > - \infty$.
It follows that $S_0, S_1, S_2, \ldots$
are contained in the half-space $H_+ (\bu) = \{ \bx \in \R^d : \bx \cdot \bu \geq c \}$.
Thus $\cH_n \subseteq H_+ (\bu)$ for all $n$, and hence $\cH_\infty \subseteq H_+(\bu)$.
Thus $\cH_\infty = \R^d$ implies $\cD_+ = \cD_- = \emptyset$, and so, by
Lemma~\ref{lem:projections-all}, (iii) implies (iv).

To show that (iv) implies (iii), we prove the contrapositive.
By Lemma~\ref{lem:r-infinity}, it suffices to suppose that $\Pr ( \cH_\infty = \R^d ) =0$. 
Since $\cl \cH_\infty$ is closed and convex, it can be written as an intersection
of hyperplanes (see e.g.~Corollary~4.1 of~\cite[p.~55]{gruber}); in particular,
if $\cl \cH_\infty$ is not the whole of $\R^d$, it is contained in a half-space
$H_-(\bu) = \{ \bx \in \R^d : \bx \cdot \bu \leq c \}$ for some $\bu \in \Sp{d-1}$
and $c \in \R$. Thus $\sup_{n \geq 0} (S_n \cdot \bu) < \infty$. In particular,
$\cP_\pm$ is not the whole of $\Sp{d-1}$. By Corollary~\ref{cor:P-plus-minus},
this implies that $\cD_\pm \neq \Sp{d-1}$.

Finally, we show that~(iv) implies (v).
Suppose that $\0 \notin \hull \cD$. Since $\hull \cD$
is closed,
this means that there is a hyperplane that separates $\0$
from $\hull \cD$, so there is a $\bu \in \Sp{d-1}$ and $c < 0$ such that
$S (\bu) =  \{ \bx \in \Sp{d-1} : \bx \cdot \bu \geq c \}$ contains no point of $\cD$.
Since $S(\bu)$ is compact, it must thus contain only finitely many of $\hat S_0, \hat S_1, \ldots$.
That is $\limsup_{n \to \infty} ( \hat S_n \cdot \bu ) \leq c$, and hence $\limsup_{n \to \infty} ( S_n \cdot \bu ) \leq 0$.
In particular,
$\cP_\pm$ is not the whole of $\Sp{d-1}$, and Corollary~\ref{cor:P-plus-minus} shows that $\cD_\pm \neq \Sp{d-1}$.
\end{proof}

\section{Some examples}
\label{sec:examples}

Let $\be_1, \ldots, \be_d$ denote the standard orthonormal basis vectors of $\R^d$.
For convenience we locate all our random walks on the integer lattice $\Z^d$, but this is not essential.
We write 
$\xi \sim \Rad$ to mean that $\Pr ( \xi = +1 ) = \Pr ( \xi = -1 ) = 1/2$ (a Rademacher distribution),
and, for $\alpha >0$, write
$\zeta \sim S(\alpha)$ to mean that $\zeta \in \Z$ has $\Pr ( \zeta \geq r ) = \Pr ( \zeta \leq - r) = \frac{1}{2} r^{-\alpha}$
for $r \in \N$. 
Our examples are constructed mostly from components that are copies of $\xi \sim \Rad$
or $\zeta \sim S(\alpha)$.

If $\xi_1, \xi_2, \ldots$ are independent copies of $\xi \sim \Rad$, then we write
$W_n = \sum_{i=1}^n \xi_i$ for the associated simple symmetric random walk (SSRW) on $\Z$.
If $\zeta_1, \zeta_2, \ldots$ are independent copies of $\zeta \sim S(\alpha)$, 
then we write $Y_n = \sum_{i=1}^n \zeta_i$. 

We recall some well-known facts about $W_n$ and $Y_n$.
The  local limit theorem for SSRW on $\Z$ (see e.g.~\cite[pp.~141--143]{dur}) says that, with $\phi$ the standard Gaussian density function,
\begin{equation}
\label{eq:llt-ssrw}
 \lim_{n \to \infty} \sup_{x \in \Z} \left| \frac{n^{1/2}}{2} \Pr \left(  W_n = 2x -n  \right) - \phi \left( \frac{2x -n}{\sqrt{n}}  \right) \right| =0 .\end{equation}
If $\alpha \in (0,1)$, then
 $Y_n$ is transient and
oscillates: $|Y_n| \to \infty$ and $Y_n$ takes both signs i.o., and, moreover
(see e.g.~Theorem~3.5 of~\cite{griffin})
\begin{equation}
\label{eq-alpha-stable}
\text{ if $\alpha \in (0,1)$, then } \liminf_{n \to \infty} n^{-1} | Y_n | = \infty, \as
\end{equation}
If $\alpha \in (0,2)$, $\alpha \neq 1$, then
$n^{-1/\alpha}Y_n$ converges in distribution to
(a constant multiple of)
a symmetric $\alpha$-stable random variable,
since
$\zeta$ is in the corresponding domain of normal attraction, with no centering
(see e.g.~Theorem~2.6.7 of~\cite{il} and~\cite[p.~580]{feller}).
If $g$ is the density of this limiting  random variable, then
Gnedenko's local limit theorem (see Theorem~4.2.1 of~\cite{il})
says that
\begin{equation}
\label{eq:llt-stable}
 \lim_{n \to \infty} \sup_{x \in \Z} \left| n^{1/\alpha} \Pr \left( Y_n = x \right)  - g ( n^{-1/\alpha} x ) \right| = 0 .\end{equation}
Note that $g$ is uniformly bounded: this follows from the inversion formula
for densities and the fact that the characteristic
function of a symmetric stable random variable is of the form $\re^{-c|t|^\alpha}$, for some $c >0$ 
(see e.g.~\cite[p.~570]{feller}).

\begin{example}
\label{ex:drift-alpha}
Suppose that $d=2$.
Let $X = \be_1 + \be_2 \zeta$ where $\zeta \sim S(\alpha)$ for $\alpha > 0$.

If $\alpha >1$ then $\Exp \| X \| < \infty$ and $\Exp X = \be_1$,
so the SLLN
implies that $S_n$ is transient with limiting direction $\be_1$,
and Proposition~\ref{lem:lim-dir} shows that $\cD = \{ \be_1 \}$.

If $\alpha\in(0,1)$, then $\| S_n \| \geq | S_n \cdot \be_1 | = n$ so the walk is again transient.
 Write $X_i = \be_1 + \be_2 \zeta_i$
where the $\zeta_i$ are independent copies of $\zeta$.
Let $Y_n = \sum_{i=1}^n \zeta_i$.  
For $j = \pm 1$, 
\[ \| \hat S_n - j \be_2 \| \leq \frac{n}{\| S_n\|} + \left| \frac{Y_n}{\| S_n\|} - j \right| .\]
By~\eqref{eq-alpha-stable} we have that $n / \| S_n \| \leq n / | Y_n | \to 0$, a.s.,
and so $\| S_n \| / | Y_n | \to 1$, a.s., and hence
\[ \left| \frac{Y_n}{\| S_n\|} - \sign (Y_n) \right| \leq \left| \frac{|Y_n|}{\| S_n \|} - 1 \right| \to 0, \as \]
It follows that, for $j = \pm 1$,
\[ \liminf_{n \to \infty} \| \hat S_n - j \be_2 \|  = \liminf_{n \to \infty} | \sign (Y_n) - j | = 0, \as \]
Hence $\{ \pm \be_2 \} \subseteq \cD$.
On the other hand, if $\bu \in \Sp{1} \setminus \{ \pm \be_2 \}$, we have $u_1 := \bu \cdot \be_1 \neq 0$,
and
\[ \liminf_{n \to \infty} \| \hat S_n - \bu \| \geq \liminf_{n \to \infty} \left| \frac{n}{\| S_n \|} - u_1 \right| = | u_1 | > 0,\]
so $\bu \notin \cD$. Thus $\cD = \{ \pm \be_2 \}$.

Finally, note that this example obviously has $\cH_\infty \neq \R^2$ (since $S_n \geq 0$ for all $n$)
while $\0 \in \hull \cD$, but $\0 \notin \Int \hull \cD$. This shows that 
(iii) and (v) of Theorem~\ref{thm:hull} are not equivalent.
\end{example}

\begin{example}
\label{ex:ber-alpha}
Suppose that $d=2$.
Let $X = \be_1 \xi + \be_2 \zeta$ where $\xi$ and $\zeta$ are independent, 
$\xi \sim \Rad$, and
$\zeta \sim S(\alpha)$ for $\alpha > 0$.

First suppose that $\alpha > 2$. Here $\Exp ( \| X\|^2 ) < \infty$
and $\Exp X =\0$, so the central limit theorem applies, and Corollary~\ref{cor:two-moments}
shows that $\cD = \Sp{1}$. Alternatively, note that the walk in this case is recurrent (see e.g.~\cite[Theorem~4.2.8]{dur})
and apply Proposition~\ref{prop:recurrence}.

Next suppose that $\alpha \in (1,2)$. In this case $\Exp X = \0$ but $\Exp ( \| X \|^2 ) = \infty$.
Here the walk is transient, as
follows from the Borel--Cantelli lemma
and the local limit theorems~\eqref{eq:llt-ssrw}
and~\eqref{eq:llt-stable},
which together show that $\Pr ( S_n = \0 ) = \Pr ( W_n = 0 ) \Pr ( Y_n = 0 ) = O ( n^{-(1/2)-(1/\alpha) } )$.
By construction, $X \eqd - X$, so Proposition~\ref{prop:plane}
shows that $\cD = \Sp{1}$.

Finally, suppose that $\alpha \in (0,1)$. Since $| S_n \cdot \be_1 | \leq n$,
a similar argument to that in Example~\ref{ex:drift-alpha}
shows that $\cD = \{ \pm \be_2 \}$. Note that this walk is transient, by Corollary~\ref{cor:not-full-sphere},
and, by Corollary~\ref{cor:zero-drift-hull}, $\Pr ( \cH_\infty = \R^d ) =1$. 
This example has $\0 \in \hull \cD$, $\0 \notin \Int \hull \cD$, and $\Pr ( \cH_\infty = \R^d ) =1$,
 showing that 
(i) and (iii) of Theorem~\ref{thm:hull} are not equivalent.
\end{example}

\begin{example}
\label{ex:four-dim}
Suppose that $d \geq 4$.
Let $X = \sum_{k=1}^{d-1} \be_k \zeta^{(k)} + \be_d \xi$ where $\xi, \zeta^{(1)}, \ldots, \zeta^{(d-1)}$ are independent, 
$\xi \sim \Rad$, and
$\zeta^{(k)} \sim S(\alpha)$ for $\alpha \in (1,2)$. This random walk has $X \eqd -X$,
$\mu = \0$, and is transient.
Let $E_d := \{ \bu \in \Sp{d-1} : \bu \cdot \be_d = 0\}$, a copy of $\Sp{d-2}$.

Recall that $C(\bu;r) = \{ \bx \in \R^d \setminus \{ \0 \} : \| \hat \bx - \bu \| < r \}$.
Fix $\eps >0$, and set
\[ B_n := \{ (x_1,x_2, \ldots, x_d ) \in \Z^d : | x_d | \leq n^{(1/2)+\eps} \} .\]
Then we have the estimate
\[
 \Pr ( S_n \in C (\bu ; r) ) \leq \Pr ( | S_n \cdot \be_d | > n^{(1/2)+\eps} ) + \sum_{\bx \in B_n \cap C (\bu ; r) } \Pr ( S_n = \bx)   . \]
Here we have from the local limit theorems~\eqref{eq:llt-ssrw} and~\eqref{eq:llt-stable} that, for some $C< \infty$,
\[ \Pr ( S_n = \bx)  = \Pr ( W_n = x_d ) \prod_{i=1}^{d-1} \Pr ( Y_n = x_i ) \leq C n^{-(d-1)/\alpha} \cdot n^{-1/2} , \]
for all $\bx \in \Z^d$. 
Standard binomial tail bounds show that for SSRW $\Pr ( | W_n | > n^{(1/2)+\eps} ) \leq C \exp ( - c n^{2\eps} )$
for constants $c>0$ and $C < \infty$.
Thus we get
  \begin{equation}
\label{eq:four-dim-eq1}
 \Pr ( S_n \in C (\bu ; r) ) \leq  C \exp ( - c n^{2\eps} ) + C \sum_{\bx \in B_n \cap C (\bu ; r) } n^{-(d-1)/\alpha} \cdot n^{-1/2}  .\end{equation}
Fix $\bu \notin E_d$,
and take $0 < r < | \bu \cdot \be_d |$.
Then any $\bx = (x_1,x_2, \ldots, x_d) \in C (\bu;r)$
has 
\[ \bigl| x_d - \| \bx \| \bu \cdot \be_d \bigr| \leq \bigl\| \bx - \| \bx \| \bu \bigr\| < r \| \bx \| .\]
Thus $( | \bu \cdot \be_d | - r ) \| \bx \| < | x_d | < ( | \bu \cdot \be_d | + r ) \| \bx \|$.
It follows that there is a constant $C < \infty$ such that $| x_i | < C | x_d |$ for all $1 \leq i \leq d-1$
and all $\bx \in C (\bu;r)$.
Hence the number of $\bx \in B_n \cap C (\bu;r)$
is at most $O( n^{(d/2)+ d  \eps} )$. Thus we obtain from~\eqref{eq:four-dim-eq1} that
\[
 \Pr ( S_n \in C (\bu ; r) ) \leq  C \exp (- cn^{2\eps }) + C n^{d \eps} n^{- (d-1)(2-\alpha)/(2\alpha)} , \]
where $C<\infty$ depends on $\bu$ and $r$, but not $\eps$.
Thus 
for any $\alpha$ satisfying
\begin{equation}
\label{eq:alpha}
 1 < \alpha < \frac{2(d-1)}{1+d} \end{equation}
we can choose $\eps >0$ small enough to ensure that $\sum_{n \geq 1} \Pr ( S_n \in C (\bu ; r) ) < \infty$.
We can find $\alpha$ satisfying~\eqref{eq:alpha} provided $d > 3$. 

Thus if we have $d \geq 4$ and $\alpha$ satisfying~\eqref{eq:alpha}, the
 Borel--Cantelli lemma shows that $\bu \notin \cD$ for any $\bu \notin E_d$, i.e., $\cD \subseteq E_d$.
On the other hand,
we have $n^{-1/\alpha} S_n$ converges in distribution to $Z = (Z_1, \ldots, Z_{d-1}, 0)$,
where the $Z_i$ are independent $\alpha$-stable random variables with $\supp Z_i = \R$.
It follows that $\supp \hat Z = E_d$, and so, by Proposition~\ref{prop:D-supp},
we conclude that
$\cD = E_d$.
\end{example}

We write $\zeta \sim S_+(\alpha)$ to mean that $\zeta \in \ZP$ has $\Pr ( \zeta \geq r ) = r^{-\alpha}$
for $r \in \N$.

\begin{example}
Let $d \in \N$ and $\alpha \in (0,1)$.
Let $X = \sum_{j=1}^k \bu_j \zeta^{(j)}$ where $k \in \N$,
the $\zeta^{(j)} \sim S_+ (\alpha)$ are independent,
and $\bu_1, \ldots, \bu_k$ are fixed vectors in $\R^d$.
For $\bz = (z_1,\ldots, z_k) \in \R^k$, set $\Lambda (\bz) := \sum_{j=1}^k z_j \bu_j$.

Write $X_i = \sum_{j=1}^k \bu_j \zeta^{(j)}_i$,
where the $\zeta^{(j)}_i$ are independent copies of $\zeta^{(j)}$,
and let $Y^{(j)}_n = \sum_{i=1}^n \zeta^{(j)}_i$. 
Then $n^{-1/\alpha} ( Y^{(1)}_n , \ldots, Y^{(k)}_n )$ converges in distribution to $( Z_1, \ldots, Z_k)$,
where $Z_1, \ldots, Z_k$ are independent, positive $\alpha$-stable random variables supported on $\RP$. 
By the continuous mapping theorem, 
$n^{-1/\alpha} S_n$ converges in distribution to $\sum_{j=1}^k \bu_j Z_j =: V$.
Since $V$ is continuous, $\Pr ( V = \0 ) = 0$, and so $\Pr ( \hat V \in \Sp{d-1} ) =1$. Thus
\[ \supp \hat V = C := C(\bu_1, \ldots, \bu_k ) := \cl \left\{ \frac{\Lambda (\bz)}{ \| \Lambda (\bz) \|}  : \bz \in \R^k, \, z_1, \ldots, z_k > 0 , \, \| \Lambda (\bz)\| > 0 \right\}. \]
Hence by Proposition~\ref{prop:D-supp}(ii) we have that $C \subseteq \cD$.

To get an inclusion in the other direction, we use the notation of Section~\ref{sec:convexity}.
We have $\supp X = \cl \{ \Lambda (\bz) : \bz \in \N^k \}$, and 
for any $\bx \in \supp X$, either $\hat \bx = \0$ (if $\bx = \0$)
or else $\hat \bx = \lim_{n \to \infty} \hat \bx_n \in \Sp{d-1}$
with $\bx_n = \Lambda (\bz_n)$ and $\bz_n \in \N^k$. It follows that
\[ \left\{ \frac{\Lambda (\bz)}{\| \Lambda (\bz) \| } : \bz \in \N^k, \, \| \Lambda (\bz ) \| > 0 \right\}
\subseteq \cX' \subseteq \{ \0 \} \cup \cl \left\{ \frac{\Lambda (\bz)}{\| \Lambda (\bz) \| } : \bz \in \N^k, \, \| \Lambda (\bz ) \| > 0 \right\} .\]
Lemma~\ref{lem:sphere-support} then shows that
\[ \cX = \cl \left\{ \frac{\Lambda (\bz)}{\| \Lambda (\bz) \| } : \bz \in \N^k, \, \| \Lambda (\bz ) \| > 0 \right\}  = \cl \left\{ \frac{\Lambda (\bz)}{\| \Lambda (\bz) \| } : \bz \in \lambda \N^k, \, \| \Lambda (\bz ) \| > 0 \right\} ,\]
for any $\lambda >0$, by scale invariance. It follows that
\[  \cX = \cl \left\{ \frac{\Lambda (\bz)}{\| \Lambda (\bz) \| } : \bz \in \Q^k, \, z_1, \ldots, z_k > 0, \, \| \Lambda (\bz ) \| > 0 \right\} .\]
Since $\Q^k$ is dense in $\R^k$, we get $\cX = C$.
Moreover,  $C$ is the closure of an s-convex set, and hence itself s-convex, by Lemma~\ref{lem:closure},
and hence $\cl \shull \cX = \shull \cX = C$, by Lemma~\ref{lem:smallest-s-convex}. Then Theorem~\ref{thm:support_bound}
confirms that $\cD = C$.
\end{example}

\section{Concluding remarks}

The Borel--Cantelli lemma shows that if for some $\eps>0$, $\sum_{n=1}^\infty \Pr ( \| \hat S_n - \bu \| < \eps ) < \infty$,
then $\Pr ( S_n \in C(\bu;\eps) \io ) = 0$, and so $\bu \notin \cD$, by Proposition~\ref{prop:D-cones}.
This is not sharp, however, as is already shown by the case of $d=1$, when, for example, $+1 \in \cD$ if and only if
$\sum_{n=1}^\infty n^{-1} \Pr ( S_n > 0 ) = \infty$~\cite[p.~415]{feller}.

\begin{problem}
\label{prob1}
Is there a criterion for $\bu \in \cD$ in terms of $\Pr ( S_n \in \, \cdot \, )$?
\end{problem}

We do not necessarily expect a simple answer to Problem~\ref{prob1}: in $d=1$, 
Kesten~(Corollary~1 of \cite[p.~1177]{kesten70}) gives a criterion for $\bx \in \cA^\alpha$  
where $\cA^\alpha$ is as defined at~\eqref{eq:A-alpha}.

Proposition~\ref{prop:plane} leaves the following question.

\begin{problem}
Suppose that $d=2$, $\Exp \| X \| < \infty$, and $\mu = \0$. Is $\cD$ always equal to $\Sp{1}$?
\end{problem}

\appendix

\section{The recurrent case}

For most of the questions in the present paper, the main interest is the transient case,
because, loosely speaking, any recurrent random walk explores all of space and hence all directions at all distances.
Proposition~\ref{prop:recurrence-coverage} is a precise version of this statement.
Recall~\cite[p.~190]{dur} that $S_n$ is \emph{recurrent}
if there is a non-empty set $\cR$ of points $\bx \in \R^d$ (the recurrent values) such that, for any $\eps >0$,
$\| S_n - \bx \| < \eps$ i.o., a.s.

\begin{proposition}
\label{prop:recurrence-coverage}
If $S_n$ is recurrent, then there exists $h >0$ such that a.s.,
 for any $\bx \in \R^d$, $S_n \in B ( \bx ; h)$ i.o.
\end{proposition}
\begin{proof}
Since $S_n$ is recurrent, the set $\cR$ of recurrent values is a closed subgroup of $\R^d$
and coincides with the set of \emph{possible values} for the walk: see~\cite[p.~190]{dur}.
Since $S_n$ is genuinely $d$-dimensional, it follows from e.g.~Theorem 21.2 of~\cite[p.~225]{br} that $\cR$ 
contains a further closed subgroup $\cR'$ of the form $H \Z^d$ where $H$ is a 
non-singular $d$ by $d$ matrix. 
Hence there exists $h >0$ such that for every $\bx \in \R^d$ there exists $\by \in \cR'$ with $\| \bx - \by \| < h /2$,
and since $\cR'$ is a countable set of recurrent values for the walk,
we have that, a.s.,  for any $\bx \in \R^d$, $S_n \in B(\bx;h)$ i.o.
\end{proof}

\section*{Acknowledgements}

Some of this work was done while
the first author was visiting Durham University in July--August 2018, supported by an 
international
study scholarship awarded by the Secretariat of Public Education of Mexico and Universidad Nacional Aut\'onoma de M\'exico (UNAM).
The authors are grateful for the comments of an anonymous referee, and also
to Nicholas Georgiou, James McRedmond, Mikhail Menshikov, and Vladislav Vysotskiy for discussions on the topic of this work.

\end{document}